\documentclass[11pt]{extarticle}

\usepackage[left=1in,right=1in,top=1in,bottom=1.3in]{geometry}
\usepackage{amsmath,amssymb,amsthm,latexsym,amsfonts,indentfirst,xcolor,bm, mathtools}
\usepackage[utf8]{inputenc}
\usepackage[english]{babel}
\usepackage{hyperref}

\usepackage{float}

\newtheorem{Theorem}{Theorem}
\newtheorem*{Theorem*}{Theorem}

\newtheorem{Corollary}[Theorem]{Corollary}

\newtheorem{Lemma}[Theorem]{Lemma}
\newtheorem{Proposition}[Theorem]{Proposition}

\newtheorem{Problem}{Problem}

\newcommand{\x}{{\mathbf x}}
\newcommand{\y}{{\mathbf y}}
\newcommand{\vv}{{\mathbf v}}
\newcommand{\w}{{\mathbf w}}
\newcommand{\z}{{\mathbf z}}
\newcommand{\R}{{\mathbb R}}
\newcommand{\Z}{{\mathbb Z}}
\newcommand{\N}{{\mathbb N}}

\newcommand{\X}{{\mathbb X}}
\newcommand{\Y}{{\mathcal Y}}
\newcommand{\g}{{\mathcal G}}
\newcommand{\M}{{\mathcal M}}
\newcommand{\B}{{\mathcal B}}
\newcommand{\lam}{\boldsymbol{\lambda}}

\newcommand{\sig}{\boldsymbol{\sigma}}

\makeatletter

\def\env@sqcases{%
	\let\@ifnextchar\new@ifnextchar
	\left\lbrack
	\def\arraystretch{1.2}%
	\array{@{}l@{\quad}l@{}}%
}
\makeatother

\begin{document}

\date{}
\title{Max-norm Ramsey Theory}
\author{N\'ora Frankl\thanks{The Open University, UK; Alfréd Rényi Institute of Mathematics, Budapest, Hungary. \newline Email:~\href{mailto:nora.frankl@open.ac.uk}{\tt nora.frankl@open.ac.uk}.} \and 
	Andrey Kupavskii\thanks{G-SCOP, Universit\'e Grenoble-Alpes, CNRS, France; MIPT, Moscow, Russia. Email:~\href{mailto:kupavskii@ya.ru}{\tt kupavskii@ya.ru}.} \and
	Arsenii Sagdeev\thanks{Alfréd Rényi Institute of Mathematics, Budapest, Hungary; MIPT, Moscow, Russia. \newline Email:~\href{mailto:sagdeevarsenii@gmail.com}{\tt sagdeevarsenii@gmail.com}.} }

\maketitle

\begin{abstract}
	Given a metric space $\M$ that contains at least two points, the chromatic number $\chi\left(\R^n_\infty, \M \right)$ is defined as the minimum number of colours needed to colour all points of an $n$-dimensional space $\R^n_\infty$ with the max-norm  such that no isometric copy of $\M$ is monochromatic. The last two authors have recently shown that the value $\chi\left(\R^n_\infty, \M \right)$ grows exponentially for all finite $\M$. In the present paper we refine this result by giving the exact value $\chi_\M$ such that $\chi\left(\R^n_\infty, \M \right) = (\chi_\M+o(1))^n$ for all ‘one-dimensional' $\M$ and for some of their Cartesian products. We also study this question for infinite $\M$. In particular, we construct an infinite $\M$ such that the chromatic number $\chi\left(\R^n_\infty, \M \right)$ tends to infinity as $n \rightarrow \infty$.
\end{abstract}

\section{Introduction}
One of the  classical problems in combinatorial geometry was posed by Nelson in 1950: what is the minimum number of colours $\chi(\R^2)$ needed to colour the Euclidean plane so that no two points of the same colour are at unit distance apart? Currently, the best known bounds are $5\le \chi(\R^2)\le 7$ \cite{degrey, Soifer}. The lower bound is a relatively recent improvement by De Grey \cite{degrey} (quickly after reproved by  Exoo and Ismailescu \cite{Exoo}) on the classical and easy-to-get bound $\chi(\R^2)\ge 4$. This bound  was first noted by Nelson, see Soifer's account of the history of the problem in \cite{Soifer}. This question spurred a great amount of subsequent research, and in particular it is the key precursor to Euclidean Ramsey Theory. Before we discuss different possible generalisations, let us introduce some notation.

Let $\X = (X, \rho_X)$ and $\Y = (Y, \rho_Y)$ be two metric spaces. A subset $Y' \subset X$ is called a {\it copy} of $\Y$ if there exists an {\it isometry} $f: Y \rightarrow Y'$, i.e., a bijection such that $\rho_Y(y_1,y_2) = \rho_X\big(f(y_1), f(y_2)\big)$ for all $y_1, y_2 \in Y$. The {\it chromatic number $\chi(\X,\Y)$ of the space $\X$ with a forbidden subspace $\Y$} is the smallest $k$ such that there is a colouring of the elements of $X$ with $k$ colours and with no monochromatic copy of $\Y$.

In these terms, $\chi(\R^2) = \chi(\R^2,I)$, where $I$ is a two-point metric space. Note that the distance between these two points does not influence the answer. The chromatic number $\chi(\R^n, I)$ of the $n$-dimensional Euclidean space was extensively studied, see the surveys by Raigorodskii \cite{Rag1,Rag2}. The best known lower bound is given by Raigorodskii \cite{Rag3} and the best known upper bound is given by Larman and Rogers \cite{Lar} (see \cite{Pros2020_LR} for an alternative proof). In particular, it grows exponentially with $n$. These results were extended to the case of the chromatic number $\chi(\R^n_p,I)$ of the $\ell_p$-space $\R^n_p$ by Raigorodskii \cite{Rag4} and the second author \cite{Kup} which, again, was shown to grow exponentially with $n$. Recall that the $\ell_p$-distance between two points ${\bf x},{\bf y}\in \R^n$ for $p\ge 1$ is given by the formulas
\begin{align*}\|{\bf x}-{\bf y}\|_p =& \big(|x_1-y_1|^p+\ldots+|x_n-y_n|^p\big)^{1/p},\\[5pt]
\|{\bf x}-{\bf y}\|_{\infty} =&\max_i |x_i-y_i|.\end{align*}
The case $p = \infty$ stands out here because of the folklore equality $\chi( \R^n_\infty, I) = 2^n$ that holds for all $n \in \mathbb{N}$.
In what follows, we shall indicate the metric in $\R^n$ using subscripts to avoid confusion.

The other important direction of generalisations was started by Erd\H os, Graham, Montgomery, Rothschild, Spencer, and Straus in three papers \cite{EGMRSS1,EGMRSS2,EGMRSS3}, and now grew into a separate area named Euclidean Ramsey Theory. As the name suggests, it deals with different Ramsey-type questions about metric subspaces of the Euclidean space $\R^n_2$ with a colouring.

A metric space $\M$, is called {\it $\ell_2$-Ramsey} if $\chi(\R^n_2,\M)\to \infty $ as $n\to \infty$.  Moreover, if this chromatic number grows exponentially, then we call $\M$ {\it exponentially $\ell_2$-Ramsey}. A central problem of Euclidean Ramsey Theory is to classify which metric spaces are Ramsey. At the moment, we know two necessary conditions for $\M$ to be $\ell_2$-Ramsey: $\M$ should be finite \cite[Theorem~19]{EGMRSS2} and lie on a sphere \cite[Theorem~13]{EGMRSS1}. Graham \cite{Graham} conjectured that these necessary conditions are also sufficient. Note that there is an interesting alternative conjecture on the characterisation of Ramsey sets due to Leader, Russell, and Walters \cite{Leader}. However, the classification problem is far from being solved, since only a few metric spaces were shown to be $\ell_2$-Ramsey. K\v{r}\'\i\v{z} \cite{Kriz1,Kriz2} and Cantwell \cite{Cant} proved that the vertex set of any regular polytope (with a natural induced metric) along with other very symmetric sets are $\ell_2$-Ramsey. Frankl and R\"odl \cite{Frankl} showed that the vertex sets of simplices and boxes are exponentially $\ell_2$-Ramsey. For the explicit bounds on the corresponding chromatic numbers see \cite{KupSagZakh,Naslund, Pros2018_ExpRams, Sag2018_CartProd}.

In a recent paper \cite{KupSag} the second and third authors started a systematic investigation of the same question for $\R^n_\infty$. They showed that, unlike the Euclidean case, all finite metric spaces are exponentially $\ell_\infty$-Ramsey.

\begin{Theorem}[Kupavskii--Sagdeev \cite{KupSag}] \label{th KupSag}
 For any finite metric space  $\M$ that contains at least two points, there exists a constant $c=c(\M)>1$ such that for all $n \in \N$ we have
 \[\chi(\R^n_\infty,\M)>c^n.\]
\end{Theorem}

However, apart from very few examples, they did not manage to determine the asymptotically tight value of the base of this exponent. In the present paper we continue this line of research.

\section{New results} \label{Sec1.1}

We split our results into three groups and state each of them in a separate subsection. The first one deals with ‘one-dimensional' metric spaces, called batons. We significantly simplify some of the arguments from \cite{KupSag} for such metric spaces, obtain asymptotically optimal values of $c(\M)$ from Theorem~\ref{th KupSag}, and uncover its connection to a certain well-studied one-dimensional number-theoretical problem. The second group of results extend the scope of these propositions to Cartesian products of arithmetic progressions. Finally, the last group deals with infinite forbidden metric spaces.

The $o$-notation in this section is with respect to $n\to \infty$ everywhere.

\subsection{Batons}

First, let us introduce some notation. Given $k \in \N$, denote $[k] = \left\{1,  2,\, \dots \,,k\right\}$ and $\left[k\right]_0 = \{0\}\cup[k]$. Given a sequence of positive reals $\lambda_1,\, \dots \, , \lambda_k$, set $\lam = (\lambda_1, \,\dots\,, \lambda_k) \in \R^k$. For all $ i \in [k]_0$, define $\sigma_i = \sum_{j=1}^{i}\lambda_j$. Note that $\sigma_0 = 0$. We call the set $\{\sigma_0, \,\dots\,, \sigma_{k}\} \subset \R$ a {\it baton} and denote it by $\B(\lam)$. If $\lambda_i \in \N$ for all $i \in [k]$, we call $\B(\lam)$ to be an {\it integer baton}. In case  $\lambda_1=\dots=\lambda_k=\lambda$, i.e., $\B(\lam)$ is an arithmetic progression with a common difference of $\lambda$, we denote $\B(\lam)$ by  $\B_k(\lambda)$ for shorthand, and if $\lambda = 1$ then we simply use $\B_k$. Batons are the main metric spaces considered in the present paper, as well as in the earlier work \cite{KupSag}.

Whenever it does not cause confusion, we use the same notation for the set and the corresponding metric space. In particular, we assume that batons $\B$ are equipped with the standard metric on $\R$ and use notation $\chi(\R^n_\infty,\B)$. We also note that, in what follows, whenever it is not explicitly stated, we assume that the spaces are equipped with the maximum metric $\ell_\infty$.

For a subset $S \subset \Z^t$, let $d(\Z^t,S)$ be the supremum of upper densities of those subsets $A \subset \Z^t$ for which for all $\x \in \Z^t$, $A$ contains neither $S+\x$ nor $-S+\x$. Several related problems on the junction of combinatorics and number theory have been extensively studied over the last few decades (see, e.g., \cite{Axe, BolJanRior2011, FranklKupSag, Hu22, LiuRob, LiuZhu, SchTul2008, SchTul2010}).

Our first theorem uncovers a relation between the values $d(\Z, \B)$ and $\chi(\R^n, \B)$ for integer batons.

\begin{Theorem} \label{th chi int baton}
	Let $k \in \N$ and $\lam=(\lambda_1,\,\dots\,,\lambda_k)$ be a sequence of positive integers. Set $\B=\B(\lam)$. Then
$$\chi(\R^n_\infty,\B)= \big(d(\Z, \B)^{-1}+o(1)\big)^n.$$
\end{Theorem}

In view of the statement, let us give some background on the value of the density $d(\Z, \B)$. First,  note that we have $d(\Z, \B) \le \frac{k}{k+1}$ for any $\B$ with $k+1$ points by a simple averaging argument. Moreover, this upper bound is tight if and only if $\B$ tiles $\Z$ (for details, see \cite{FranklKupSag}). In particular, for $\B=\B_k$, the set $A=[k]+(k+1)\Z$ is of density $\frac{k}{k+1}$, and contains neither $\B_k+x$ nor $-\B_k+x$ for any $x \in \Z$. As for the lower bound, it is an easy modification of some classical results in the field to show that there exists a constant $c\sim1.577$ such that $d(\Z, \B(\lam)) \ge 1-\frac{c+\ln k}{k}$ for all sequences $\lam=(\lambda_1,\,\dots\,,\lambda_k)$ of $k$ positive integers, see e.g. \cite[Corrolary~3.2]{BolJanRior2011}. In particular, these two bounds ensure that the value $d(\Z, \B(\lam))$ tends to $1$ as $k \rightarrow \infty$ with `almost linear' speed. 

Furthermore, let us note that there is an algorithm that calculates the exact value of $d(\Z, \B)$ for any given integer baton $\B$ with time complexity $O(4^{\text{diam\,} \B})$, see \cite[Remark~5.4]{BolJanRior2011}. So, the statement of Theorem~\ref{th chi int baton} is rather explicit. However, the problem of finding the closed-form expression for $d(\Z, \B(\lambda_1, \,\dots\,,\lambda_k))$ seems to be very hard in general. For $|\B|=2$, i.e. in the case when $\B$ is a two-point metric space $I$, it is almost obvious that $d(\Z, I) = \frac{1}{2}$. Surprisingly, the case of a three-point baton $\B(\lambda_1, \lambda_2)$ is already quite tricky. Schmidt and Tuller~\cite{SchTul2008} conjectured the explicit form of the function $d(\Z, \B(\lambda_1, \lambda_2))$ in terms of $\lambda_1$ and $\lambda_2$. We managed to confirm their conjecture in~\cite{FranklKupSag}. For $|\B| = 4$ it seems challenging even to formulate a conjecture concerning the form of the corresponding density function.

In this paper we also prove an extension of Theorem~\ref{th chi int baton} that states that for an arbitrary (not necessarily integer) baton $\B$, the limit $\chi(\R^n_\infty,\B)^{1/n}$ exists, as well as `determine' its value. See Section~\ref{sec4} for the precise statement. Here we only present one of its neat special cases.

\begin{Theorem} \label{cor lin indep baton}
	Let $k \in \N$ be an integer and $\lam=(\lambda_1,\,\dots\,,\lambda_k)$ be a sequence of real numbers linearly independent over $\Z$.  Then
	$$\chi(\R^n_\infty,\B(\lam)) =\left(\frac{k+1}{k}+o(1)\right)^n.$$
\end{Theorem}

Observe that the last two results yield that the limit $\chi(\R^n_\infty,\B(\lam))^{1/n}$ is not continuous as a function of $\lam$. Indeed, Theorem~\ref{th chi int baton} implies that $\chi(\R^n_\infty,\B(1,2)) =\left(\frac{5}{3}+o(1)\right)^n$, while it follows from Theorem~\ref{cor lin indep baton} that $\chi(\R^n_\infty,\B(1,\lambda)) =\left(\frac{3}{2}+o(1)\right)^n$ for any irrational $\lambda$ arbitrarily close to $2$.

\subsection{Cartesian products}

For any two metric spaces $\mathcal X = \left( X, \rho_X\right)$ and $\Y = \left( Y, \rho_Y\right)$ we define their {\it Cartesian product} $\mathcal X \times \Y$ as the metric space $\left( X\times Y, \rho\right)$, where
\begin{equation*}
\rho\left(\left(x_1, y_1 \right), \left( x_2, y_2 \right) \right) = \max\left\{ \rho_X\left( x_1, x_2 \right), \rho_Y\left( y_1, y_2 \right)  \right\}
\end{equation*}
for all $x_1,x_2 \in X$ and $y_1,y_2 \in Y$. As usual, we denote $\mathcal X \times \mathcal X$ by $\mathcal X^2$ and so on.

Given $m \in \N$, the distance between any two points of $\B_1^m$ is clearly equal to $1$. So, one can think of $\B_1^m$ as of a `regular simplex'. Given $n \in \N$, it is not hard to see that $\chi(\R_\infty^n, \B_1^m) = \lceil \frac{2^n}{2^m-1} \rceil$, which generalises the aforementioned equality $\chi(\R_\infty^n, \B_1) = \chi(\R_\infty^n) = 2^n$.

More generally, Theorem~\ref{th chi int baton} implies that for all $k \ge 2$, we have $\chi(\R_\infty^n, \B_k) = \left( \frac{k+1}{k}+o(1)\right)^n$. The next theorem shows that the same asymptotic equality holds for Cartesian powers of $\B_k$ as well.

\begin{Theorem} \label{Chi Bkm}
	Given $k \ge 2$ and  $m \in \N$, one has
	\begin{equation*}
		(1+o(1))\frac{(m-1)!k^{m-1}}{n^{m-1}}\Big(\frac{k+1}{k}\Big)^n \le \chi(\R_\infty^n, \B_k^m) \le (1+o(1))(\ln k)n\Big(\frac{k+1}{k}\Big)^n.
	\end{equation*}
	In particular, $\chi(\R_\infty^n, \B_k^m) = \left(\frac{k+1}{k}+o(1)\right)^n$.
\end{Theorem}
We get such a sharp estimate for the chromatic number because we can determine the independence number of a certain relevant hypergraph {\it exactly}, see Section~\ref{sec5} for details.

It would be very interesting to understand the behaviour of the chromatic number under taking Cartesian products. Clearly, for all $n \in \N$ and for all metric spaces $\Y_1, \Y_2$, the chromatic number $\chi(\R^n_\infty,\Y_1\times \Y_2)$ does not exceed the minimum of $\chi(\R^n_\infty,\Y_1)$ and $\chi(\R^n_\infty,\Y_2)$, since both $\Y_1$ and $\Y_2$ may be considered as subspaces of their Cartesian product. It is interesting that this trivial upper bound is sometimes asymptotically tight, as Theorem~\ref{Chi Bkm} shows. We think that this is always the case for batons and state it in the following form. 

\begin{Problem}\label{prb111}
	Let $\Y_1, \Y_2$ be two arbitrary one-dimensional metric spaces and $c_1, c_2$ be positive reals such that $\chi(\R^n_\infty,\Y_i) = (c_i+o(1))^n$, $i=1,2$. Set $c=\min\{c_1,c_2\}$. Is it true that
	$$\chi(\R^n_\infty,\Y_1\times \Y_2)\ge (c+o(1))^n \, ?$$
\end{Problem}

We suspect that a similar result should hold for Cartesian products of more than two batons as well. In this paper we resolve this conjecture in the particular case when all these batons are arithmetic progressions that may have different common differences.

\begin{Theorem} \label{th chi grid}
	Let $k,m \in \N$ and $\lambda_1,\,\dots\,,\lambda_m$ be positive reals. Then
	\begin{equation*}
		\chi\big(\R_\infty^n, \B_{k}(\lambda_1)\times\dots\times \B_{k}(\lambda_m)\big) = \left(\frac{k+1}{k} +o(1)\right)^n.
	\end{equation*}
\end{Theorem}

The following special case of the last result deserves a separate mention.
\begin{Corollary} \label{cor chi box}
	Let $\mathcal H$ be a hyperrectangle, i.e., a Cartesian product of several two-point sets. Then
	\begin{equation*}
		\chi(\R_\infty^n, \mathcal H) = \left(2+o(1)\right)^n.
	\end{equation*}
\end{Corollary}

Observe that Theorem~\ref{th chi grid} allows to work with Cartesian products of arithmetic progressions that may also have different number of terms. Indeed, fix $m \in \N$, and two sequences of positive numbers $k_1,\,\dots\,,k_m \in \N$ and $\lambda_1,\,\dots\,,\lambda_m \in \R$. Set $\M = \B_{k_1}(\lambda_1) \times\dots\times \B_{k_m}(\lambda_m)$. Without loss of generality, assume that $k \coloneqq \max_{i}\{k_i\} = k_1$. Clearly, $\B_k(\lambda_1) \subset \M \subset \M',$ where $\M' = \B_{k}(\lambda_1) \times\dots\times \B_{k}(\lambda_m)$. Thus, for all $n \in \N$, we have
\begin{equation*}
	\chi(\R^n_\infty, \M') \le \chi(\R^n_\infty, \M) \le \chi(\R^n_\infty, \B_k(\lambda_1)),
\end{equation*}
and Theorems~\ref{th chi grid} and~\ref{Chi Bkm} imply that both lower and upper bounds are of the form $\left(\frac{k+1}{k}+o(1)\right)^n$.

\subsection{Infinite metric spaces}

Finally, we turn our attention to infinite forbidden metric spaces. In the Euclidean case, the situation is simple: for any  dimension $n$ and any infinite metric space $\mathcal M$, there exists a $2$-colouring of $\R_2^n$ with no monochromatic copy of $\mathcal M$. In fact, a much stronger statement is true, see \cite[Theorem~19]{EGMRSS2}.

The picture in the $\ell_\infty$-norm turns out to be much more complicated. Observe that, unlike in the finite case, not all infinite $\M$ can be embedded into $\R_{\infty}^n$ for some $n \in \N$. However, we may assume without loss of generality that this is the case, since otherwise we simply have $\chi(\R_\infty^n, \M)=1$.

\begin{Proposition} \label{prop upper for infinite intro}
	Given $n \in \N$ and an infinite $\M \subset \R_{\infty}^n$, we have $\chi(\R_\infty^n, \M) \le n+1.$
\end{Proposition}
Surprisingly, this bound is tight for certain geometric progressions. For a real number $q>0$, let us define the following infinite set of reals $\mathcal G(q)$ equipped with the standard metric on $\R$:
$$\g(q) \coloneqq \{0,1,1+q,1+q+q^2,1+q+q^2+q^3,\,\ldots\}.$$

\begin{Theorem} \label{cor geom baton intro}
	Given $n \in \N$, there exists $q_0>0$ such that for any $0<q\le q_0$, we have $$\chi(\R_\infty^n, \g(q)) = n+1.$$
\end{Theorem}

This result, however, does not show the existence of a {\it single} infinite metric space with chromatic number tending to infinity, or, in other words, it does not imply that an infinite $\ell_\infty$-Ramsey metric space exists. This is guaranteed by the next theorem.

\begin{Theorem}\label{thminf2 intro}
	There exists $q_0>0$ such that for any $0<q\le q_0$ and any $n \in \N$, we have
	$$\chi(\R_{\infty}^n,\g(q)) > \log_{3}n.$$
\end{Theorem}

It seems to be a challenging and interesting open problem to improve this logarithmic dependence to linear.

\vspace{5mm}

We organise the remainder of the paper as follows. In Section~\ref{sec2}, we state and prove some preliminary results. In Section~\ref{sec3}, we prove Theorem~\ref{th chi int baton}. In Section~\ref{sec4}, we generalise it to arbitrary, not necessarily integer batons. In Section~\ref{sec5}, we consider Cartesian products of arithmetic progressions. In Section~\ref{sec6} we study infinite forbidden metric spaces. Finally, in Section~\ref{sec7}, we make some further comments and state more open problems.

\section{Preliminaries}\label{sec2}

\subsection{Independence numbers}
Let $\X = (X, \rho_X)$ and $\Y = (Y, \rho_Y)$ be two metric spaces. A subset $X'$ of $X$ is called \emph{$\Y$-free} if it does not contain an isometric copy of $\Y$. In case $|X|, |Y| < \infty$ we define the {\it independence number $\alpha(\X,\Y)$ of the space $\X$ with a forbidden subspace $\Y$} as the maximum cardinality of an $\Y$-free subset $X' \subset X$. By using the pigeonhole principle, it is clear that
\begin{equation}\label{eqpighole}
	\chi(\X,\Y) \ge \frac{|X|}{\alpha(\X,\Y)}.
\end{equation}

Let $S, S'$ be two subsets of $\R^t$. We call $S'$ a {\it translation of $S$} if for some $\x \in \R^t$ we have $S' = S+\x$. We call $S'$ a {\it reflection} of $S$ if for some $\x \in \R^t$ we have $S' = -S+\x$. For a set $S$, we call a subset $A \subset \R^t$ to be  {\it $S$-tr-free} if $A$ contains neither a translation nor a reflection  of $S$. (Here, `tr' stands for `translation and reflection'.) Observe that if $t=1$, the subset $A$ is $S$-tr-free if and only if it does not contain an isometric copy of $S$, i.e., it is $S$-free. For any finite subset $X\subset \R^t$, let us define the {\it independence number} $\alpha_{tr}(X, S)$ as
\begin{equation*}
	\alpha_{tr}(X, S) = \max_{Y \subset X}\{|Y|: Y \mbox{ is $S$-tr-free}\}.
\end{equation*}
(Note that we use subscript `tr' to distinguish this $\alpha$ from the $\alpha$ in the isometry sense.)

Given $m \in \N$, we also denote by $\alpha_{tr}(\Z_m^t, S)$ the independence number of the group $\Z_m^t$, i.e., let
\begin{equation*}
	\alpha_{tr}(\Z_m^t, S) = \max_{Y \subset \Z^t_m}\{|Y|: Y+m\Z^t \mbox{ is $S$-tr-free}\}.
\end{equation*}

Now assume that $S$ is a finite subset of $\Z^t$. Recall that the value $d(\Z^t,S)$ is defined as the supremum of upper densities over all $S$-tr-free subsets $A\subset \Z^t$. We need the following natural proposition (see, e.g., \cite[Lemma~6.1]{BolJanRior2011}) that establishes the connection between $d(\Z^t,S)$ and the independence numbers defined above.

\begin{Proposition} \label{prop limit}
	For all $t \in \N$ and for all finite $S \subset \Z^t$, both limits
	\begin{equation*}
		\lim_{m\rightarrow \infty} \frac{\alpha_{tr}([m]^t, S)}{m^t} \ \ \ \ \mbox{and} \ \ \ \ \lim_{m\rightarrow \infty}\frac{\alpha_{tr}(\Z_m^t, S)}{m^t}
	\end{equation*}
	exist and are equal to $d(\Z^t, S)$.
\end{Proposition}

\subsection{Integer metric spaces}

We call an arbitrary metric space $\M=(M,\rho_M)$ {\it integer} if $\rho_M(x,y) \in \N$ for all distinct $x,y \in M$. Observe that this notion trivially generalises the aforementioned definition of an integer baton. Let us give the following general and simple proposition about integer metric spaces that we believe is of independent interest.

\begin{Proposition} \label{prop ZN=RN}
	Let $\M=(M,\rho_M)$ be an integer metric space. Then for all $n \in \N$ we have
	\begin{equation*}
		\chi(\R_\infty^n, \M) = \chi(\Z_\infty^n, \M).
	\end{equation*}
\end{Proposition}

\begin{proof}
Since $\Z^n \subset \R^n$, it follows from the definition that $\chi(\Z_\infty^n, \M) \le \chi(\R_\infty^n, \M)$.
	
To prove the reverse inequality we consider a proper colouring of $\Z^n$ with $\chi(\Z_\infty^n, \M)$ colours and then assign to each unit `cell' the colour of its `bottom left' vertex. That is, to each $\x=(x_1,\,\dots\,,x_n)\in \R^n$ we assign the colour of its {\it rounding} $\lfloor\x\rfloor = (\lfloor x_1 \rfloor,\,\dots\,,\lfloor x_n \rfloor) \in \Z^n$. It remains only to check that this gives us a proper colouring of $\R^n$.

Note that for any two reals $x$ and $y$, we have $$\lfloor x \rfloor-\lfloor y \rfloor-1<x-y< \lfloor x \rfloor-\lfloor y \rfloor+1,$$ which implies that $x-y = \lfloor x \rfloor-\lfloor y \rfloor$ whenever $x-y$ is an integer. Applying this over all coordinates, it is easy to see that the same holds for vectors $\x,\y\in \R^n$: whenever we have $\|\x-\y\|_\infty\in \N$, we also have $$\|\x-\y\|_\infty = \big\|\lfloor\x\rfloor-\lfloor\y\rfloor\big\|_\infty.$$

Now we assume that there is a monochromatic isometric copy $M' \subset \R^n$ of $\M$. Using the property of integer parts above, we conclude that $\{\lfloor\x\rfloor : \x \in M'\} \subset \Z^n$ is also a monochromatic copy of $\M$, a contradiction.
\end{proof}

\subsection{Baton-free Cartesian powers}

The following lemma gives a necessary and sufficient condition for an arbitrary set of points of $\R^n$ to be a copy of a given baton. One of its simple corollaries illustrates the connection between the values $\chi(\R_\infty^n, \B)$ and $d(\Z,\B)$. We will use the notation $\x = (x_1,\,\dots\,,x_n)$ for points $\x\in \mathbb{R}^n$.

\begin{Lemma} \label{lemma baton embed}
Let $k,n \in \N$ and $\lam=(\lambda_1,\,\dots\,,\lambda_k)$ be a sequence of positive reals. The sequence $(\x^0,\,\ldots\,, \x^k)$ of points in $\R^n$ is isometric to $\B(\lam)$ (in that order) if and only if the following two conditions hold. First, there exists an $i\in[n]$ such that the sequence $(x_i^0,\,\ldots\,, x_i^k) \subset \R$ is isometric to $\B(\lam)$ (i.e., is either a translation or a reflection of $\B(\lam)$). Second, for any $j\in [n]$ and any $r \in [k]$ we have  $|x_j^{r}-x_j^{r-1}| \le |x_i^r-x_i^{r-1}| = \lambda_r.$
\end{Lemma}

\begin{proof}
Let the sequence $(\x^0,\,\ldots\,, \x^k)$ of points in $\R^n$ be an isometric copy of $\B(\lam)$. Using similar notation as before,  set $\sigma_r = \sum_{t=1}^{r} \lambda_t$ for all $r \in [k]_0$. Let $i$ be a coordinate such that $|x^k_i-x^0_i| =\sigma_k$ (there must exist at least one such coordinate, since $\|\x^k-\x^0\|_\infty = \sigma_{k}$). Then it should be clear that, for each $r=1,\,\ldots\,, k-1$, we have a unique choice of $x_i^r$ so that $|x^r_i-x^0_i| \le \|\x^r-\x^0\|_\infty = \sigma_r$ and $|x^k_i-x^r_i| \le \|\x^k-\x^r\|_\infty = \sigma_k-\sigma_r$. Moreover, both inequalities must hold with equality in that case. This concludes the proof of the first part. The second part of the statement is trivial, since for any $j \in [n]$ and any $r \in [k]$, we clearly have $|x_j^r-x_j^{r-1}| \le \|\x^r-\x^{r-1}\|_\infty = \lambda_r = |x_i^r-x_i^{r-1}|$.

To prove the opposite direction, let us assume that the sequence $(\x^0,\,\ldots\,, \x^k)$ of points in $\R^n$ satisfies both conditions of the lemma. Given $0 \le l < r \le k$, observe that the first condition implies that $|x_i^r-x_i^l| = \sigma_r-\sigma_l$. Moreover, from the second condition and the triangle inequality it follows that
$$|x_j^r-x_j^l| \le \sum_{t=l+1}^{r} |x_j^t-x_j^{t-1}| \le \sum_{t=l+1}^{r} |x_i^t-x_i^{t-1}| = \sum_{t=l+1}^{r}(\sigma_t-\sigma_{t-1}) = \sigma_r-\sigma_l$$
for all $j \in [n]$. So, we have $\|\x^r-\x^l\|_{\infty} = |x_i^r-x_i^l| = \sigma_r-\sigma_l$. Thus, the sequence $(\x^0,\,\ldots\,, \x^k)$ is indeed an isometric copy of $\B(\lam)$.
\end{proof}

The following corollary is immediate from the lemma.

\begin{Corollary} \label{cor power free}
	Let $k \in \N$ and $\lam=(\lambda_1,\,\dots\,,\lambda_k)$ be a sequence of positive reals. Set $\B=\B(\lam)$. Then for all $n \in \N$ and for all $A \subset \R$, the set $A^n \subset \R^n$ is $\B$-free if and only if $A$ is $\B$-tr-free.
\end{Corollary}

\subsection{$\R_{\infty}^n$ as a poset}

Here we use a standard (see, e.g., \cite{BlokhWil, GKS, Swan}) yet powerful point of view on the $n$-dimensional space $\R_{\infty}^n$ as on a poset to obtain some corollaries. First, let us recall the necessary basic notions.

A {\it partially ordered set}, or {\it poset} for shorthand, is a pair $P = (X,\preceq)$, where $X$ is a set and $\preceq$ is a reflexive, transitive and antisymmetric binary relation over it. If $x\preceq y$ or $y\preceq x$, then we say that $x,y$ are {\it comparable}, and we call them {\it incomparable} otherwise. A set of pairwise comparable elements is called a {\it chain}, and a set of pairwise incomparable elements is called an {\it antichain}. The sizes of the largest chain and antichain are called the {\it length $\ell(P)$} and {\it width $w(P)$} of the poset respectively. The following theorem is one of the cornerstones of poset theory.

\begin{Theorem}[Dilworth's theorem \cite{Dilworth}] \label{Dilworth's theorem}
	Let $P$  be an arbitrary finite poset. Then the width $w(P)$ of $P$ is equal to the minimum number of chains that altogether cover $P$. In particular, $|P| \le \ell(P)w(P)$.
\end{Theorem}

We define a binary relation $\preceq$ over $\R_{\infty}^n$ by
\begin{equation} \label{eq poset}
	\x \preceq \y \mbox{ if and only if } \|\y-\x\|_\infty = y_n-x_n
\end{equation}
for all $\x, \y \in \R^n$. One can easily check that $(\R_{\infty}^n,\preceq)$ is a poset. Indeed, reflexivity and antisymmetry are immediate from the definition, and transitivity follows from the triangle inequality.

Now we use this partial order to prove that any sufficiently large subset of $\R_{\infty}^n$ with the induced maximum metric contains an isometric copy of some `long' baton.

\begin{Proposition} \label{prop no batons}
	Given $k,n \in \N$ and $A\subset \R_{\infty}^n$, if $A$ contains no isometric copies of $(k+1)$-point batons, then $|A| \le k^n$. Moreover, this bound is tight.
\end{Proposition}
\begin{proof}
	The proof is by induction in $n$. As any $k+1$ points on a line form a translation of some $(k+1)$-point baton, the $n=1$ case is trivial.
	To prove the induction step, we consider a poset $(A,\preceq)$ with the partial order defined above.
	
	Observe that any chain $\x^0\preceq\dots\preceq\x^k$ of size $k+1$ would be isometric to a $(k+1)$-point baton. (To see this, apply Lemma~\ref{lemma baton embed} with $i=n$.) Therefore, the length of the poset does not exceed $k$.
	Next, we bound the width. Given $\x=(x_1,\,\dots\,,x_n) \in A$, we define $\x' = (x_1,\,\dots\,,x_{n-1}) \in \R^{n-1}$. It is not hard to see that $\|\x-\y\|_\infty = \|\x'-\y'\|_\infty$ for all incomparable $\x,\y \in A$. Therefore, projecting an antichain on the hyperplane induced by the first $n-1$ coordinates is an injection, and it gives a subset of $\R_{\infty}^{n-1}$ with no  isometric copies of $(k+1)$-point batons. Thus, the cardinality of this projection is at most $k^{n-1}$ by the induction hypothesis, and so is the width of $(A,\preceq)$. Now Dilworth's theorem yields that $|A| \le k \cdot k^{n-1} = k^n$.
	
	To see the tightness of this upper bound, it is sufficient to consider $[k]^n \subset \R_{\infty}^n$. Indeed, the size of the projection of $[k]^n$ on each axis equals $k$. Hence, Lemma~\ref{lemma baton embed} implies that $[k]^n$ does not contain isometric copies of $(k+1)$-point batons.
\end{proof}

This general Ramsey-type proposition has the following neat geometric corollary.

\begin{Corollary} \label{cor box embed}
	Given $n \in \N$, let $\mathcal H \subset \R_{\infty}^n$ be a set of all $2^n$ vertices of an $n$-dimensional axis parallel hyperrectangle. Then $\mathcal H$ cannot be isometrically embedded into $\R_{\infty}^{n-1}$.
\end{Corollary}
\begin{proof}
	Again, the size of the projection of $\mathcal H$ on each axis equals $2$. Hence, Lemma~\ref{lemma baton embed} implies that $\mathcal H$ does not contain isometric copies of $3$-point batons. However, Proposition~\ref{prop no batons} shows\footnote{Actually, this proves even more: no $2^{n-1}+1$ points of $\mathcal H$ can be isometrically embedded into $\R_{\infty}^{n-1}$.} that the cardinality of any subset of $\R_{\infty}^{n-1}$ with the same property does not exceed $2^{n-1} < 2^n$. 
\end{proof}

\section{Integer batons: proof of Theorem \ref{th chi int baton}}\label{sec3}

We start with the following crucial proposition.

\begin{Proposition} \label{prop indep power}
	Let $\B$ be an arbitrary baton and $X\subset \R$ be an arbitrary finite set. Then for all $n \in \N$, we have $\alpha(X^n,\B) = \alpha_{tr}(X,\B)^n$.
\end{Proposition}

\begin{proof} We start with the lower bound $\alpha(X^n,\B) \ge \alpha_{tr}(X,\B)^n$. It follows from Corollary~\ref{cor power free} that for any $\B$-tr-free subset $A\subset X$ of cardinality $|A| = \alpha_{tr}(X,\B)$ the product $A^n \subset X^n$ is $\B$-free.
	
	The proof of the upper bound is by induction on $n$. The $n=1$ case is trivial, since any isometric copy of $\B$ on the line is either a translation or a reflection of $\B$. So, we turn to the induction step.
	
	Let $A\subset X^n$ be a $\B$-free subset of cardinality $|A| = \alpha(X^n,\B)$. Consider a poset $(A,\preceq)$ with the partial order defined by \eqref{eq poset}.
	
	First, observe that if $C\subset A$ is a chain with respect to $\preceq$, then $|C|\leq \alpha(X,\B) = \alpha_{tr}(X,\B)$. Indeed, if $C$ is a chain then the $\ell_\infty$-distance between any two of its points depends only on their last coordinates. Since $C$ is $\B$-free, then so is the projection of $C$ on the last coordinate. Moreover, this projection is clearly an injection.
	
	Now let us suppose that $C\subset A$ is an antichain with respect to $\preceq$. Given $\x=(x_1,\,\dots\,,x_n) \in X^n$, we define $\x' = (x_1,\,\dots\,,x_{n-1}) \in X^{n-1}$. It is not hard to see that $\|\x-\y\|_\infty = \|\x'-\y'\|_\infty$ for all incomparable $\x,\y \in A$. Therefore, projecting $C$ on the hyperplane induced by the first $n-1$ coordinates is an injection, and it gives a $\B$-free subset of $X^{n-1}$. Thus, the cardinality of this projection is at most $\alpha(X^{n-1}, \B)$ by definition. Hence,  we conclude by induction that $|C| \le \alpha(X^{n-1}, \B) \le \alpha_{tr}(X, \B)^{n-1}$.
	
	Using Dilworth's theorem, we finally  get that
	\begin{equation*}
		\alpha(X^n, \B) = |A| \le \alpha_{tr}(X, \B)\cdot\alpha_{tr}(X, \B)^{n-1} = \alpha_{tr}(X, \B)^{n}.\qedhere
	\end{equation*}
\end{proof}

\begin{Corollary} \label{cor chi int baton lower}
	Let $k \in \N$ and $\lam=(\lambda_1,\,\dots\,,\lambda_k)$ be a sequence of positive integers. Set $\B=\B(\lam)$. Then for all $n \in \N$, we have
	\begin{equation*}
		\chi(\R_\infty^n,\B) \ge d(\Z, \B)^{-n}.
	\end{equation*}
\end{Corollary}

\begin{proof}
	Let $n \in \N$ be fixed. Given $m \in \N$, we apply Proposition \ref{prop indep power} to the set $X=[m]$ and get that $\alpha([m]^n, \B) = \alpha_{tr}([m], \B)^{n}$. Now it follows from \eqref{eqpighole} that
	\begin{equation*}
		\chi(\R_\infty^n,\B) \ge \frac{\big|[m]^n\big|}{\alpha([m]^n, \B)} = \left(\frac{m}{\alpha_{tr}([m], \B)}\right)^n.
	\end{equation*}
	Since the last inequality holds for all $m \in \N$, Proposition \ref{prop limit} now provides the desired lower bound by letting $m$ tend to infinity.
\end{proof}

We now move on to the upper bound. The short proof of the following proposition is based on the classic Erd\H{o}s--Rogers-type argument (see \cite{ErRog1962}) and can be found, e.g., in \cite{BolJanRior2011}. However, we spell it out below for completeness.

\begin{Proposition} \label{prop translates}
	Let $m, n$ be two positive integers and $X \subset \Z_m^n$ be a nonempty set. Then the additive group $\Z_m^n$ can be covered by
	$ m^n \!\left( \frac{1+\ln |X|}{|X|}\right)$ translates of $X$.
\end{Proposition}
\begin{proof}
	
	Set, with foresight, $r = m^n \cdot \frac{\ln |X|}{|X|}$ and let $\x_1, \,\dots\, ,\x_r$ be elements of $\Z_m^n$ chosen uniformly and independently at random. Let $Y$ be a random set consisting of all elements of $\Z_m^n$ that are not covered by $X+\x_i$ for all $i \in [r]$.
	
	Given $\y \in \Z_m^n$ and $i \in [r]$, it is easy to see that
	\begin{equation*}
		\Pr [\y \in X+\x_i] = \frac{|X|}{m^n}.
	\end{equation*}
	Therefore, it follows from the mutual independence of the choice of $\x_i$'s and the linearity of the expectation that
	\begin{equation*}
		\mathrm{E}[|Y|] = m^n\left(1-\frac{|X|}{m^n}\right)^r \le m^n\left(1-\frac{|X|}{m^n}\right)^{\frac{m^n}{|X|}\ln|X|} < m^n e^{-\ln|X|} = \frac{m^n}{|X|}.
	\end{equation*}
	
	Hence, there is a way to fix $\x_1, \,\dots\, ,\x_r \in \Z_m^n$ such that $|Y| \le \frac{m^n}{|X|}$. Each element of $Y$ can be easily covered by some individual translation of $X$. Thus, the total number of translates of $X$ that is required to cover $\Z_m^n$ is at most
	\begin{equation*}
		r+|Y| \le m^n \cdot \frac{\ln |X|}{|X|} + \frac{m^n}{|X|} \le  m^n\left( \frac{1+\ln |X|}{|X|}\right). \qedhere
	\end{equation*}
\end{proof}

\begin{Corollary} \label{cor chi int baton upper}
	Let $k \in \N$ and $\lam=(\lambda_1,\,\dots\,,\lambda_k)$ be a sequence of positive integers. Set $\B=\B(\lam)$. Then, as $n \rightarrow \infty$, we have
	\begin{equation*}
		\chi(\R_\infty^n,\B) \le \left(d(\Z, \B)^{-1}+o(1)\right)^n.
	\end{equation*}

\end{Corollary}

\begin{proof}
	Fix $m \in \N$ and consider a subset $A\subset\Z_m$ of cardinality $|A|=\alpha_{tr}(\Z_m,\B)$ such that $A+m\Z$ is $\B$-tr-free. (Note that the existence of such $A$ trivially follows from the definition of $\alpha_{tr}(\Z_m,\B)$.) Proposition \ref{prop translates} implies that $\Z_m^n$ can be covered by
	\begin{equation*}
		\left(\frac{m}{\alpha_{tr}(\Z_m,\B)}+o(1)\right)^n
	\end{equation*}
	translates of $A^n$ as $n \rightarrow \infty$. It is easy to see that $\Z^n$ can be covered by the same number of translates of $(A+m\Z)^n$. Moreover, it follows from Corollary \ref{cor power free} that $(A+m\Z)^n \subset \Z^n$ is $\B$-free. Assigning an individual colour to each translate, we conclude that
	\begin{equation*} \label{eq cor chi int baton upper}
		\chi(\Z_\infty^n,\B) \le \left(\frac{m}{\alpha_{tr}(\Z_m,\B)}+o(1)\right)^n.
	\end{equation*}
	The statement now follows from  Propositions~\ref{prop limit} and~\ref{prop ZN=RN} via a standard limit argument.
\end{proof}

Corollaries~\ref{cor chi int baton lower} and~\ref{cor chi int baton upper} together give the statement of Theorem \ref{th chi int baton}.

\section{Arbitrary batons}\label{sec4}

Let $\B$ be an arbitrary, not necessarily integer baton. In the present section we prove that the limit $(\chi(\R^n,\B))^{1/n}$ exists and determine its value.
Although this result generalises Theorem~\ref{th chi int baton}, we decided to give them separately because the argument is much cleaner in the integer case. 

We shall need some more notation. Given $k \in \N$ and a vector of positive reals $\lam = (\lambda_1,\,\dots\,,\lambda_k)$, we denote the abelian group that the coordinates generate by $\langle \lam \rangle$, i.e.,
\begin{equation*}
	\langle \lam \rangle = \{a_1\lambda_1+\dots+a_k\lambda_k: a_1,\,\dots\,, a_k \in \Z\}.
\end{equation*}
Since $\langle \lam \rangle$ is trivially a torsion-free group, the fundamental theorem of finitely generated abelian groups states that $\langle \lam \rangle \cong \Z^t$ for some $t \in [k]$.  Let $\varphi: \langle \lam \rangle \rightarrow \Z^t$ be one such isomorphism. Set $\B=\B(\lam)$ and $S = \varphi(\B)\subset \Z^t$ (in the last expression, we think of $\B$ as of the subset of $\langle \lam \rangle$).  Recall that the value $d(\Z^t, S)$ is defined as the supremum of the upper densities over all subsets $A \subset \Z^t$ that contains neither a translation nor a reflection of $S$. The following theorem is the main result of this section.

\begin{Theorem} \label{th chi gen baton}
	In the above notation, we have $\chi(\R^n_\infty,\B) = \big(d(\Z^t,S)^{-1}+o(1))^n$.
\end{Theorem}

Again, let us make some comments upon the statement of the last theorem to clarify it. First, it is not hard to see that the last result generalises Theorem~\ref{th chi int baton}. Indeed, if all coordinates of $\lam$ are positive integers, then $\langle \lam \rangle \cong \Z$, i.e., $t=1$. Since multiplying all coefficients of $\lam$ by the same number changes neither $\chi(\R^n_\infty,\B)$ nor $d(\Z,\B)$, we can assume without loss of generality  that $\gcd(\lambda_1,\,\dots\,,\lambda_k)=1$. In this case one can take the identity function as the isomorphism $\varphi: \langle \lam \rangle \rightarrow \Z$ providing the equality $d(\Z,\B) = d(\Z^t,S)$.

Second, recall that in Section~\ref{Sec1.1}, we discussed that in case $t=1$, there is a constant $c\sim1.577$ such that for all $S \subset \Z^t$ of cardinality $k+1$, we have
\begin{equation} \label{bound}
	1-\frac{c+\ln k}{k}\le d(\Z^t,S) \le \frac{k}{k+1},
\end{equation}
Schmidt and Tuller~\cite[Corollary~1.5]{SchTul2010} showed that these bounds hold in the general $t \ge 1$ case as well (with the same constant $c$ for all $t$). So, one can get a relatively good estimation of $\chi(\R^n_\infty,\B)$ based only on the cardinality of $\B$. Another similarity between the general and the one-dimensional cases here is the following. For all $t \in \N$, it is known (again, see~\cite{FranklKupSag}) that the upper bound from \eqref{bound} is tight if and only if $S$ tiles $\Z^t$.

Let us also note that, in contrast to the one-dimensional case, for any given $t \ge 2$, it is unknown if there is an algorithm that, given a finite $S \subset \Z^t$, calculates the exact value of $d(\Z^t,S)$. However, in case $t=2$, Bhattacharya \cite{Bhat2020} showed the existence of an algorithm that, given a subset $S \subset \Z^2$ of cardinality $k+1$, verifies if $S$ tiles $\Z^2$, i.e., if $d(\Z^2,S) = \frac{k}{k+1}$. See Szegedy~\cite{Sze1998} or Greenfeld and Tao \cite{GreenTao2020, GreenTao2021, GreenTao2022, GreenTao2023} for partial results and further discussion on this problem in higher dimensions. Nevertheless, observe that it follows from Proposition~\ref{prop limit} that, given $\varepsilon > 0$, the value $d(\Z^t,S)$ can be algorithmically estimated with an error of less than $\varepsilon$. Besides, one can solve this problem ad hoc in some special cases. For instance, let us deduce Theorem~\ref{cor lin indep baton} from Theorem~\ref{th chi gen baton}.

\begin{proof}[Proof of Theorem~\ref{cor lin indep baton}]
Fix $k \in \N$ and a sequence $\lam=(\lambda_1,\,\dots\,,\lambda_k)$ of real numbers linearly independent over $\Z$. In this case, it is clear that $\langle \lam \rangle \cong \Z^k$ and a function $\varphi: \langle \lam \rangle \rightarrow \Z^k$ given by
\begin{equation*}
	\varphi (a_1\lambda_1+\dots+a_k\lambda_k) = (a_1,\,\dots\,,a_k)
\end{equation*}
is an isomorphism. Set $\B=\B(\lam)$, $S = \varphi(\B)\subset \Z^k$, and
\begin{equation*}
	A = \{ \x \in \Z^k : x_1+\dots+x_k \not\equiv 0\mod{k+1} \}.
\end{equation*}
Clearly, the density of $A$ is $\frac{k}{k+1}$. Besides, $A$ contains neither a translation nor a reflection $S'\subset\Z^k$ of $S$, since each such $S'$ covers all residues modulo $k+1$ after summing up the coordinates of its points. Hence, $d(\Z^k,S) \ge \frac{k}{k+1}$. The opposite inequality $d(\Z^k,S) \le  \frac{k}{k+1}$ follows from \eqref{bound}. Now Theorem~\ref{th chi gen baton} concludes the proof.
\end{proof}

\subsection{Proof of Theorem \ref{th chi gen baton}}

\begin{Lemma} \label{lemma B T}
	In the notation of Theorem \ref{th chi gen baton}, a subset $A \subset \langle \lam \rangle$ is $\B$-tr-free if and only if $\varphi(A)\subset\Z^t$ is $S$-tr-free. In particular, for all finite $X \subset \langle \lam \rangle$, we have $\alpha_{tr}(X,\B) = \alpha_{tr}(\varphi(X), S)$.
\end{Lemma}

\begin{proof}
	Assume that $A \subset \langle \lam \rangle $ is not $\B$-tr-free, i.e., for some $\sigma \in \{-1,+1\}$ and $x_0 \in A$, we have $\sigma\B+x_0 \subset A$. By linearity of the isomorphism $\varphi$, we have that $$\varphi(\sigma\B+x_0) = \sigma\varphi(\B)+\varphi(x_0) = \sigma S+\varphi(x_0),$$
	and thus $\varphi(A)$ is not $S$-tr-free. The proof of the other direction is analogous.
\end{proof}

\begin{Corollary} \label{cor chi gen baton lower}
	In the notation of Theorem \ref{th chi gen baton}, for all $n \in \N$, we have
	\begin{equation*}
		\chi(\R_\infty^n,\B) \ge d(\Z^t, S)^{-n}.
	\end{equation*}
\end{Corollary}

\begin{proof}
	Let $n \in \N$ be fixed. Given $m \in \N$, set $X = \varphi^{-1}([m]^t) \subset \langle \lam \rangle$. It follows from Lemma~\ref{lemma B T} that $\alpha_{tr}(X,\B) = \alpha_{tr}([m]^t, S)$. Moreover, Proposition~\ref{prop indep power} implies that $\alpha(X^n,\B)=\alpha_{tr}(X,\B)^n.$ Now one can easily deduce from \eqref{eqpighole} that
	\begin{equation*}
		\chi(\R_\infty^n,\B) \ge \frac{|X^n|}{\alpha(X^n,\B)} = \frac{\big|[m]^t\big|^n}{\alpha_{tr}([m]^t, S)^n} = \left(\frac{m^t}{\alpha_{tr}([m]^t, S)}\right)^n.
	\end{equation*}
	Since the last inequality holds for all $m \in \N$, Proposition \ref{prop limit} now provides the desired lower bound by letting $m$ tend to infinity.
\end{proof}

\begin{Proposition} \label{cor chi gen baton upper}
	In the notation of Theorem \ref{th chi gen baton}, as $n \rightarrow \infty$, we have
	\begin{equation*}
		\chi(\R_\infty^n,\B) \le \left(d(\Z^t, S)^{-1}+o(1)\right)^n.
	\end{equation*}
\end{Proposition}

\begin{proof}
	Fix $m \in \N$ and consider a subset $X\subset\Z_m^t$ of cardinality $|X|=\alpha_{tr}(\Z_m^t,S)$ such that $X+m\Z^t$ is $S$-tr-free (recall that the existence of such $X$ trivially follows from the definition of $\alpha_{tr}(\Z_m^t,S)$). Proposition \ref{prop translates} implies that $\Z_m^{tn}$ can be covered by $r=r(n)$ translates of $X^n$, where
	\begin{equation} \label{eq cor chi gen baton upper 1}
		r=\left(\frac{m^t}{\alpha_{tr}(\Z_m^t,S)}+o(1)\right)^n
	\end{equation}
	as $n \rightarrow \infty$. It is easy to see that $\Z^{tn}$ can be covered by the same number of translates of $(X+m\Z^{t})^n$, i.e., there are $\x_1,\,\dots\,,\x_r \in \Z^{tn}$ such that
	\begin{equation} \label{eq cor chi gen baton upper 2}
		\Z^{tn} = \bigcup_{i=1}^{r}\Big((X+m\Z^{t})^n+\x_i\Big).
	\end{equation}
	
	Set $A = \varphi^{-1}(X+m\Z^t) \subset \langle \lam \rangle$. Given $n \in \N$, we extend the isomorphism $\varphi: \langle \lam \rangle \rightarrow \Z^t$ component-wise to the isomorphism $\varphi: \langle \lam \rangle^n \rightarrow \Z^{tn}$ between two abelian groups with naturally defined component-wise addition. For all $i \in [r]$ we also set $\y_i = \varphi^{-1}(\x_i) \in \langle \lam \rangle^n$. Now it follows from \eqref{eq cor chi gen baton upper 2} that
	\begin{align} \label{eq cor chi gen baton upper 3}
		\langle \lam \rangle^n = \varphi^{-1}(\Z^{tn}) =&\ \bigcup_{i=1}^{r} \varphi^{-1}\Big((X+m\Z^{t})^n+\x_i\Big) = \nonumber \\
		=&\ \bigcup_{i=1}^{r} \Big( \big(\varphi^{-1}(X+m\Z^t)\big)^n +\varphi^{-1}(\x_i) \Big) = \bigcup_{i=1}^{r} (A^n+\y_i).
	\end{align}
	
	Recall that the set $X+m\Z^t$ is $S$-tr-free by assumption. Hence, it follows from Lemma~\ref{lemma B T} that $A = \varphi^{-1}(X+m\Z^t)$ is $\B$-tr-free. Hence, Corollary~\ref{cor power free} implies that $A^n$ is $\B$-free. Now by assigning an individual colour to each translation $A^n+\y_i$, $i \in [r]$, we conclude from \eqref{eq cor chi gen baton upper 3} that $$\chi(\langle \lam \rangle^n, \B) \le r,$$ where $r=r(n)$ is from \eqref{eq cor chi gen baton upper 1}.
	
	Now we want to show that \eqref{eq cor chi gen baton upper 3} implies a stronger inequality $\chi(\R_\infty^n, \B) \le r$. A similar step in the proof of Corollary \ref{cor chi int baton upper} was immediate by using Proposition~\ref{prop ZN=RN}. Unfortunately, we do not have an analogue of this proposition that would guarantee that the values  $\chi(\R_\infty^n, \B)$ and $\chi(\langle \lam \rangle^n, \B)$ are actually equal in our case. Nevertheless, we can bypass this difficulty.

Consider $\langle \lam \rangle$ as a subgroup of $\R$. Let $\Omega \subset \R$ be such that the set \mbox{$\{\langle \lam \rangle+\omega: \omega \in \Omega\}$} is the set of all cosets of $\langle \lam \rangle$ in $\R$, i.e.,
	\begin{equation} \label{eq cor chi gen baton upper 4}
		\R = \bigsqcup_{\omega \in \Omega} (\langle \lam \rangle+\omega),
	\end{equation}
	where $\sqcup$ stands for a disjoint union. 
It simply follows from \eqref{eq cor chi gen baton upper 4} that
	\begin{equation} \label{eq cor chi gen baton upper 5}
		\R^n = \bigsqcup_{\bm{\omega} \in \Omega^n} (\langle \lam \rangle^n+\bm{\omega}).
	\end{equation}
Now we combine \eqref{eq cor chi gen baton upper 3} with \eqref{eq cor chi gen baton upper 5} to conclude that
	\begin{equation*}
		\R^n = \bigsqcup_{\bm{\omega} \in \Omega^n} \left( \bigcup_{i=1}^{r} (A^n+\y_i)+\bm{\omega}\right) = \bigcup_{i=1}^{r} \left( \bigsqcup_{\bm{\omega} \in \Omega^n} (A^n+\bm{\omega})+\y_i\right).
	\end{equation*}

	Thus, to prove the inequality $\chi(\R_\infty^n, \B) \le r$ it remains to show that the set
	$$\bigsqcup_{\bm{\omega} \in \Omega^n} (A^n+\bm{\omega}) = \left(\bigsqcup_{\omega \in \Omega} (A+\omega)\right)^n \subset \R^n$$
	is $\B$-free. By Corollary~\ref{cor power free}, it is sufficient to show that $\bigsqcup_{\omega \in \Omega} (A+\omega) \subset \R$ is $\B$-tr-free. Each term $A+\omega$ here is clearly $\B$-tr-free as it is a translation of the set $A$ which is $\B$-tr-free by construction. Moreover, since the distance between any two points of $\B$ belongs to $\langle \lam \rangle$, we conclude that each translation or reflection of $\B$ in $\R$ lies entirely within some coset $\langle \lam \rangle+\omega$, $\omega \in \Omega$. Thus, the union  $\bigsqcup_{\omega \in \Omega} (A+\omega)$ is indeed $\B$-tr-free.
	
 Finally, we get that
	\begin{equation*} \label{eq cor chi gen baton upper 6}
		\chi(\R_\infty^n,\B) \le r= \left(\frac{m^t}{\alpha_{tr}(\Z_m^t,S)} +o(1)\right)^n
	\end{equation*}
	as $n \rightarrow \infty$. The rest of the proof is a straightforward limiting argument that uses Proposition~\ref{prop limit}, and we omit it. 
\end{proof}

Corollary~\ref{cor chi gen baton lower} and Proposition~\ref{cor chi gen baton upper} combined together give the statement of Theorem~\ref{th chi gen baton}.

\section{Cartesian products}\label{sec5}

The main result of the first part of this section is the following theorem, that given  $k,m,n \in \N$, determines the maximum cardinality of a $\B_k^m$-free subset of $\B_k^n$ {\it precisely}. Note that an isometric copy of $\B_k^m$ inside $\B_k^n$ can be naturally considered as an `$m$-dimensional subspace'. So, one may think of our next result as of a max-norm analog of the multidimensional density Hales--Jewett theorem \cite{FurKat, HalJew, Pol1} which turned out to be incredibly simpler than the original statement.

\begin{Theorem} \label{Bkm3}
	For all $k \ge 2$ and for all $m,n \in \N$, one has
	\begin{equation*}
	\alpha(\B_k^n, \B_k^m) = \sum_{i=0}^{m-1} \binom{n}{i}k^{n-i}.
	\end{equation*}
\end{Theorem}

Observe that the case $k=1$, which is not covered by the last result, is rather degenerate. Indeed, given $m,n \in \N$, both $\B_1^n$ and $\B_1^m$ are regular unit simplices with $2^n$ and $2^m$ vertices respectively, and thus it is easy to see that $\alpha(\B_1^n, \B_1^m) = \min\{2^n, 2^m-1\}$.

We can use this theorem as follows to estimate the value $\chi(\R_\infty^n, \B_k^m)$ and prove Theorem~\ref{Chi Bkm}.

\begin{proof}[Proof of Theorem~\ref{Chi Bkm}]
Let us consider a Cartesian power $\B_k^n$ as a subset $[k]_0^n$ of $\R_\infty^n$ with induced metric. Now the lower bound is straightforward from Theorem~\ref{Bkm3} and \eqref{eqpighole}. Indeed, for all fixed $k,m \in \N$, we get that 
$$\chi(\R_\infty^n, \B_k^m) \ge \frac{|\B_k^n|}{\alpha(\B_k^n, \B_k^m)} = \frac{(k+1)^n}{(1+o(1)) \binom{n}{m-1}k^{n-m+1}}= (1+o(1))\frac{(m-1)!k^{m-1}}{n^{m-1}} \Big(\frac{k+1}{k}\Big)^n$$
as $n \to \infty$, since the sum $\sum_{i=0}^{m-1} \binom{n}{i}k^{n-i}$ is equal to its last term asymptotically.

The upper bound follows from the inequality $\chi(\R_\infty^n, \B_k^m) \le \chi(\R_\infty^n, \B_k)$ and the upper bound on $\chi(\R_\infty^n, \B_k)$ that was proved in \cite[Theorem~9]{KupSag}.
\end{proof}

The rest of this section is organised as follows. First, we prove Theorem~\ref{Bkm3}. Then we generalise the arguments from this proof and deduce Theorem~\ref{th chi grid}.

\subsection{Proof of Theorem~\ref{Bkm3}}

During this proof let us consider a Cartesian power $\B_k^n$ as a subset $[k]_0^n$ of $\R_\infty^n$ with the induced metric for convenience. We begin with the upper bound.

\begin{Proposition} \label{Bkm1}
	For all $k,m,n \in \N$, one has
	\begin{equation} \label{Bkm1 1}
		\alpha([k]_0^n, \B_k^m) \le \sum_{i=0}^{m-1} \binom{n}{i}k^{n-i}.
	\end{equation}
\end{Proposition}

\begin{proof}
Let $k$ be a fixed integer. The proof is by induction on $m$ and $n$. More precisely, we use the assumption for $(m-1,n-1)$ and $(m,n-1)$ to prove it for $(m,n)$. Since the $0$-ary Cartesian power of a given set is a singleton by definition, it is clear that $\alpha([k]_0^0, \B_k^m)=1$ and $\alpha([k]_0^n, \B_k^0)=0$ for all $m,n \in \N$. So, inequality \eqref{Bkm1 1} is vacuous for $m=0$ or $n=0$. We use these trivial cases as the base of the induction.

Next we turn to the induction step. Let $A \subset [k]_0^n$ be a subset of cardinality $|A| = \alpha([k]_0^n, \B_k^m)$ with no copy of $\B_k^m$. For a vector $\x=\left(x_1,\, \dots\,, x_n\right)\in [k]_0^n$ define its {\it head} $h( \x )=x_n$ and  {\it tail} $t(\x)=(x_1,\, \dots\,, x_{n-1})$. Given $\y \in [k]_0^{n-1}$, set $H(\y) = \{i \in [k]_0: (\y, i) \in A\}$.
	
Let $Y = \{\y \in [k]_0^{n-1}: H(\y) = [k]_0\}$. One can see that $|Y| \le \alpha([k]_0^{n-1}, \B_k^{m-1})$. Indeed, if $|Y| > \alpha([k]_0^{n-1}, \B_k^{m-1})$, then by definition there is a subset $S \subset Y$ that is a copy of $\B_k^{m-1}$. Thus, the subset $\{(\y, i): \y \in S, i \in [k]_0\} \subset A$ is a copy of $\B_k^{m}$, a contradiction. Hence, by induction we have that
	\begin{equation} \label{Bkm1 2}
	|Y| \le \sum_{i=0}^{m-2} \binom{n-1}{i}k^{n-i-1}.
	\end{equation}
(In the case $m=1$ we simply have $|Y|=0$.)
Put $Y'\coloneqq \{(\y,0): \y\in Y\} \subset A$ and  $A' \coloneqq A\!\setminus\! Y'.$ Note that $|A| =|Y'|+|A'| = |Y|+|A'|$.

The next step is to bound $|A'|$. Let us define the head-moving function $f:A'\rightarrow[k]_0^n$ that increases the last coordinate of a vector by $1$ `whenever possible' as follows. (We note that it has some similarity to the shifting or compression technique in extremal set theory, see, e.g., \cite{Fra1}.) For a vector $\x\in A'$, set
\begin{equation*}
f(\x) =
\begin{cases}
\left(t(\x),h(\x)+1\right) & \mbox{if } \exists\, j \in [k]_0\!\setminus\! H(t(\x)) \mbox{ such that } j > h(\x);\\
\x & \mbox{otherwise.}
\end{cases}
\end{equation*}

We show that $f$ is an injection. Indeed, it is clear that $f(\mathbf{x}^1)\neq f(\mathbf{x}^2)$ for all $\mathbf{x}^1, \mathbf{x}^2$ such that $t(\mathbf{x}^1) \neq t(\mathbf{x}^2)$ or $|h(\mathbf{x}^1)-h(\mathbf{x}^2)| \ge 2$. Thus, let us consider $\mathbf{y} \in [k]_0^{n-1}$, $i \in [k]$, such that both $(\mathbf{y},i-1)$ and $(\mathbf{y},i)$ belong to $A'$. It is not hard to see that either $f(\mathbf{y},i-1) = (\mathbf{y},i)$ and $f(\mathbf{y},i) = (\mathbf{y},i+1),$ or $f(\mathbf{y},i-1) = (\mathbf{y},i-1)$ and $f(\mathbf{y},i) = (\mathbf{y},i)$. In both cases we have $f(\mathbf{y},i-1)\neq f(\mathbf{y},i)$. Thus, $f$ is really an injection.

Recall that $H(\y)\neq [k]_0$ by construction for all $\y \in [k]_0^{n-1}$ such that $(\y,0)\in A'$. Thus, it is not hard to check that $h(f(\x)) \neq 0$ for all $\x \in A'$. Hence, putting
\begin{equation*}
	A_i' = \{f(\x): \x \in A' \mbox{ and } h(f(\x)) =i \},
\end{equation*}
we have
\begin{equation} \label{Bkm1 4}
	 |A'| = |A_1'|+\dots+|A_k'|.
\end{equation}

Note that the distance between any two points in $A_i'$ is the same as the distance between their preimages. Indeed, if two points fall into the same $A_i'$ then their difference in the last coordinate was at most $1$ before, and could not be the unique maximal difference. Since $A'$ is $\B_k^m$-free by construction, we conclude that each $A_i'$ is $\B_k^m$-free as well. Moreover, one can naturally consider $A_i'$ as a subset of $[k]_0^{n-1}$, because all the vectors from $A_i'$ have the same last coordinate (equal to~$i$). Thus, we can apply induction to each $A_i'$ to conclude that 
	\begin{equation*} \label{Bkm1 5}
		|A_i'| \le \alpha([k]_0^{n-1}, \B_k^m) \le \sum_{i=0}^{m-1} \binom{n-1}{i}k^{n-i-1}.
	\end{equation*}
Combining this with \eqref{Bkm1 2} and 	\eqref{Bkm1 4}, we get
	\begin{align*}
	\alpha([k]_0^n, \B_k^m) = |A| = |Y|+|A'|\le& \left(\sum_{i=0}^{m-2} \binom{n-1}{i}k^{n-i-1}\right) + k\left(\sum_{i=0}^{m-1} \binom{n-1}{i}k^{n-i-1}\right)\\
=& \left(\sum_{i=1}^{m-1} \binom{n-1}{i-1}k^{n-i}\right) + \left(\sum_{i=0}^{m-1} \binom{n-1}{i}k^{n-i}\right)\\
=& \sum_{i=0}^{m-1} \binom{n}{i}k^{n-i}. \qedhere
	\end{align*}
\end{proof}

In what follows, we prove the lower bound in Theorem~\ref{Bkm3} that, unlike the upper bound from Proposition~\ref{Bkm1}, holds only for $k \ge 2$. In general, it is difficult to grasp how the possible isometric copies of $\B^m_k$ might look like in $[k]_0^n$, but the following lemma establishes one necessary condition.

\begin{Lemma} \label{k=2t}
Let $k \ge 2$ and $m,n \in \N$ be such that $m \le n$. Let $g:[k]_0^m \rightarrow [k]_0^n$ be an isometry and put $t = \lfloor k/2\rfloor$. Then there is $\z \in [k]_0^m$ such that $g(\z) \in [k]_0^n$ has at least $m$ coordinates equal to $t$.
\end{Lemma}

\begin{proof}
	For each $1\le i \le m$ we consider $\x^i = (t, \,\dots\,, t, 0, t, \,\dots\,, t), \y^i = (t, \,\dots\,, t, k, t, \,\dots\,, t) \in [k]_0^m$, where the only different coordinate for $\x^i$ and $\y^i$ has index equal to $i$. Since $\|\y^i-\x^i\|_\infty = k$ and $g$ is an isometry, there exists $1\le j_i \le n$ such that either $g(\x^i)_{j_i} = 0$ and $g(\y^i)_{j_i} = k$, or $g(\x^i)_{j_i} = k$ and $g(\y^i)_{j_i} = 0$. In the former case we set $z_i=t$, while in the latter case we set $z_i = k-t$ and switch the roles of $\x^i$ and $\y^i$. Put $\z = (z_1,\,\ldots\,, z_m) \in [k]_0^m$.
	
	One may easily check that $\|\y^i-\x^{i'}\|_\infty  \le \max\{t,k-t\} < k$ for any $i\ne i'$. Thus, we conclude that $j_{i} \neq j_{i'}$ for all distinct $i$ and $i'$. 
	
	Given $1\le i \le m$, we have $|y^i_i-z_i|=k-t$ and $|z_i-x^i_i|=t$ by construction. Moreover, for all $i'\ne i$, it is not hard to see that the value $|y^i_{i'}-z_{i'}|=|z_{i'}-x^i_{i'}|=|z_{i'}-t|$ is either $0$ or $1$. Thus, we conclude that $\|\y^i-\z\|_{\infty} = k-t$ and $\|\z-\x^i\|_{\infty} = t$.
	
	Since $g$ is an isometry, we get that $\|g(\y^i)-g(\z)\|_\infty =k-t$ and  $\|g(\z)-g(\x^i)\|_\infty = t$ for all $1\le i\le m$. In particular, $|k-g(\z)_{j_i}| = |g(\y^i)_{j_i}-g(\z)_{j_i}| \le k-t$ and $|g(\z)_{j_i}-0|=|g(\z)_{j_i}-g(\x^i)_{j_i}| \le t$. Hence, we have $g(\z)_{j_i} = t$ for all $1\le i \le m$.
\end{proof}
Note that the choice of $t$ was arbitrary in the proof, since we could have proceeded in a similar way for other integers $t$ such that $\frac{1}{4}< \frac{t}{k} <  \frac{3}{4}$. However, the statement of Lemma~\ref{k=2t} is not valid for $t=0$ and $t=k$.

The following proposition concludes the proof of Theorem~\ref{Bkm3}.

\begin{Proposition} \label{Bkm2}
	For all $k \ge 2$ and for all $m,n \in \N$, one has
	\begin{equation*} \label{Bkm2 1}
	\alpha([k]_0^n, \B_k^m) \ge \sum_{i=0}^{m-1} \binom{n}{i}k^{n-i}.
	\end{equation*}
\end{Proposition}

\begin{proof}
Modulo the lemma above, the proof is straightforward. Set $t = \lfloor \frac{k}{2} \rfloor$. Let $A$ be a subset of $[k]_0^n$ defined by
	\begin{equation*}
		A = \big\{\x \in [k]_0^n : |\{i : x_i=t\}| < m\big\}.
	\end{equation*}
 Lemma \ref{k=2t} guarantees that $A$ does not contain a copy of $\B_k^m$. Thus,
	\begin{equation*}
	\alpha([k]_0^n, \B_k^m) \ge |A| = \sum_{i=0}^{m-1} \binom{n}{i}k^{n-i}. \qedhere
	\end{equation*}
\end{proof}

\subsection{Proof of Theorem \ref{th chi grid}}

Recall that, given a positive $\lambda \in \R$ and an integer $k \in \N$, we denote a baton $\B(\lambda,\,\,\dots\,,\lambda)$ with $k$ lambdas by $\B_k(\lambda)$.

First, observe that the upper bound in the desired asymptotic equality
\begin{equation*}
	\chi\big(\R_\infty^n, \B_{k}(\lambda_1)\times\dots\times \B_{k}(\lambda_m)\big) = \left(\frac{k+1}{k}+o(1)\right)^n
\end{equation*}
is immediate from 
\begin{equation*}
	\chi\big(\R_\infty^n, \B_{k}(\lambda_1)\times\dots\times \B_{k}(\lambda_m)\big) \le \chi\big(\R_\infty^n, \B_{k}(\lambda_1)\big) = \chi\big(\R_\infty^n, \B_{k}\big) = \left( \frac{k+1}{k}+o(1)\right)^n,
\end{equation*}
where the last equality is from Theorem~\ref{th chi int baton}. The key additional ingredient in the proof of the lower bound is to study the independence number
$\alpha\big(X^n, \B_{k}(\lambda_1)\times\dots\times \B_{k}(\lambda_m)\big)$ for a particular finite subset $X \subset \R$.

Given a natural $k \in \N$ and a positive $\lambda \in \R$, we call a subset $X\subset\R$ to be {\it $\B_k(\lambda)$-tessellatable} if and only if $\B_k(\lambda)$ tiles $X$, i.e. one can partition $X$ into translates of $\B_k(\lambda)$. (Here, think of $\B_k(\lambda)$ as of a set.)

\begin{Lemma} \label{lemma tessel}
	Given $m$, $k_1,\,\dots\,,k_m \in \N$, and positive  reals $\lambda_1\le\dots\le\lambda_m$, suppose that a finite subset $X\subset \R$ is $\B_{k_1}(\lambda_1)$-tessellatable. Then for all $n \in \N$, we have
\begin{align*}
	\alpha\big(X^n, \B_{k_1}(\lambda_1)\times\dots\times \B_{k_m}(\lambda_m) \big) \le k_1\cdot\frac{|X|}{k_1+1}\,\cdot\, & \alpha\big(X^{n-1}, \B_{k_1}(\lambda_1)\times\dots\times \B_{k_m}(\lambda_m) \big)\\
	 + \frac{|X|}{k_1+1} \,\cdot\, & \alpha\big(X^{n-1}, \B_{k_2}(\lambda_2)\times\dots\times \B_{k_m}(\lambda_m) \big).
\end{align*}
\end{Lemma}
The proof of this lemma generalises the argument from the proof of Proposition~\ref{Bkm1}.
\begin{proof}	
	Let $A \subset X^n$ be a subset of cardinality
	\begin{equation*}
		|A| = \alpha\big(X^n, \B_{k_1}(\lambda_1)\times\dots\times \B_{k_m}(\lambda_m)\big)
	\end{equation*}
	with no copy of $\B_{k_1}(\lambda_1)\times\dots\times \B_{k_m}(\lambda_m)$.
	
	Since $X$ is $\B_{k_1}(\lambda_1)$-tessellatable, we can partition it into $r=\frac{|X|}{k_1+1}$ sets $X_1,\,\dots\,,X_r$, where for all $j \in [r]$ there is $h_j\in X$ such that $X_j = \B_{k_1}(\lambda_1)+h_j$. We use the notions of {\it head} and {\it tail}, introduced in the proof of Proposition~\ref{Bkm1}. For a vector $\y \in X^{n-1}$, set $H_j(\y) = \{x \in X_j: (\y, x) \in A\}$. In other words, $H_j(\y)$ are the heads from $X_j$ that extend the tail $\y$.
	
	Given $j \in [r]$, set $Y_j = \{\y \in X^{n-1}: H_j(\y) = X_j\}$. One can see that
	\begin{equation} \label{eq lemma tessel 1}
		|Y_j| \le  \alpha\big(X^{n-1}, \B_{k_2}(\lambda_2)\times\dots\times \B_{k_m}(\lambda_m) \big).
	\end{equation}
	Indeed, if \eqref{eq lemma tessel 1} does not hold, then by the definition of the independence number there is a subset $S\subset Y_j$ that is copy of $\B_{k_2}(\lambda_2)\times\dots\times \B_{k_m}(\lambda_m)$. In this case,  the subset $\{(\y, x): \y \in S, x \in X_j\} \subset A$ is a copy of $\B_{k_1}(\lambda_1)\times\dots\times \B_{k_m}(\lambda_m)$, a contradiction.
	
	Put $Y' \coloneqq \{(\y,h_j): j \in [r], \y \in Y_j\}$ and define $A' = A\!\setminus\! Y'$. 
	It is clear that
	\begin{equation} \label{eq lemma tessel 2}
		|A| = |A'|+|Y'| = |A'|+ \sum_{j=1}^{r}|Y_j|.
	\end{equation}
	Using \eqref{eq lemma tessel 1} and \eqref{eq lemma tessel 2}, it remains only to give an upper bound on $|A'|$ of the form
\begin{equation}\label{eq lemma tessel 23}
	|A'|\le k_1\cdot \frac{|X|}{k_1+1}\cdot\alpha\big(X^{n-1}, \B_{k_1}(\lambda_1)\times\dots\times \B_{k_m}(\lambda_m) \big).
\end{equation}

To do so we define a head-moving function $f:A'\rightarrow X^n$, resembling the one in the proof of Proposition~\ref{Bkm1}. It keeps the head of each element inside its $X_j$ but increases it by $\lambda_1$ `whenever possible'. Formally, for all $j \in [r]$ and for a vector $\x\in A'$ such that $h(\x) \in X_j$, we set
	\begin{equation*}
		f(\x) =
		\begin{cases}
			\left(t(\x),h(\x)+\lambda_1\right) & \mbox{if } \exists\, x \in X_j\!\setminus\! H_j\left(t(\x)\right) \mbox{ such that } x > h(\x);\\
			\x & \mbox{otherwise.}
		\end{cases}
	\end{equation*}

	One can easily repeat the same argument from the proof of Proposition~\ref{Bkm1} to verify the following two properties of this function. First, $f$ is an injection. Second, for all $\x \in A'$ and $j\in [r]$, we have $h(f(\x)) \neq h_j$. Thus, we can split the image of $f$ into $|X|-r = k_1\cdot\frac{|X|}{k_1+1}$ layers based on their heads:
	\begin{equation} \label{eq lemma tessel 3}
		|A'| = \sum_{x \in X\setminus\{h_1,\dots,h_r\}} |A_x'|,
	\end{equation}
	where
	\begin{equation*}
		A_x' = \{f(\x): \x \in A' \mbox{ and } h(f(\x)) =x \}.
	\end{equation*}

As in the case of Proposition~\ref{Bkm1}, we verify that no $A_x'$ contains a copy of $\B_{k_1}(\lambda_1)\times\dots\times \B_{k_m}(\lambda_m)$. Indeed, assume the contrary, i.e., that there exists $S \subset A_x'$ that is a copy of $\B_{k_1}(\lambda_1)\times\dots\times \B_{k_m}(\lambda_m)$. On the one hand, for all $(\y,x) \in S$ the preimage $f^{-1}(\y,x)$ is either $(\y,x)$ or $(\y,x-\lambda_1)$. On the other hand, since $\lambda_i\ge \lambda_1$ for any $i\in [m]$, the distance $\|\y^1-\y^2\|_{\infty}$ between any two distinct points $(\y^1,x), (\y^2,x)$ in $S$ is at least $\lambda_1$. Thus, $f$ could not change the distances between these points. Therefore, the preimage $f^{-1}(S) \subset A$ is also a copy  of $\B_{k_1}(\lambda_1)\times\dots\times \B_{k_m}(\lambda_m)$, a contradiction.

So, for all $x$ the set $A_x'$ is indeed $\B_{k_1}(\lambda_1)\times\dots\times \B_{k_m}(\lambda_m)$-free, and thus
	\begin{equation} \label{eq lemma tessel 4}
		|A_x'| \le \alpha\big(X^{n-1}, \B_{k_1}(\lambda_1)\times\dots\times \B_{k_m}(\lambda_m) \big).
	\end{equation}
Now we combine inequalities \eqref{eq lemma tessel 3} and \eqref{eq lemma tessel 4} to get \eqref{eq lemma tessel 23} and finish the proof.
\end{proof}

Solving the recurrence relation in the statement of the lemma, we get the following corollary.

\begin{Corollary}\label{cor11} Given $k,m \in \N$, and  positive  reals $\lambda_1\le\dots\le\lambda_m$, suppose that a finite subset $X\subset \R$ is $\B_{k}(\lambda_i)$-tessellatable for all $i \in [m]$. Then for all $n \in \N$, we have
$$\alpha\big(X^n, \B_{k}(\lambda_1)\times\dots\times \B_{k}(\lambda_m) \big) \le \Big(\frac{|X|}{k+1}\Big)^n\cdot \sum_{i=0}^{m-1}{n\choose i}k^{n-i}.$$
\end{Corollary}

\begin{proof} As in the proof of Proposition~\ref{Bkm1}, for any fixed $k\in \N$ and $X \subset \R$, we use induction on $m$ and $n$. More precisely, we use the assumption for $(m-1,n-1)$ and $(m,n-1)$ to prove the statement for $(m,n)$. Since the $0$-ary Cartesian product is a singleton by definition, it is clear that the desired inequality is vacuous for $m=0$ or $n=0$. We use these trivial cases as the base of induction.

So, we turn to the induction step. Assume that both $m$ and $n$ are positive. Using Lemma~\ref{lemma tessel} and the induction hypothesis  we have
{\footnotesize	\begin{align*}
\alpha\big(X^n, \B_{k}(\lambda_1)\times\dots\times \B_{k}(\lambda_m) \big) \le &\frac{k\cdot|X|}{k+1}\cdot\alpha\big(X^{n-1}, \B_{k}(\lambda_1)\times\dots\times \B_{k}(\lambda_m) \big)+	  \frac{|X|}{k+1} \cdot \alpha\big(X^{n-1}, \B_{k}(\lambda_2)\times\dots\times \B_{k}(\lambda_m) \big)\\
\le &\frac{k\cdot|X|}{k+1}\cdot\Big(\frac{|X|}{k+1}\Big)^{n-1} \cdot \sum_{i=0}^{m-1}{n-1\choose i}k^{n-1-i}+ \frac{|X|}{k+1} \cdot \Big(\frac{|X|}{k+1}\Big)^{n-1} \cdot \sum_{i=0}^{m-2}{n-1\choose i}k^{n-1-i}\\
=& \Big(\frac{|X|}{k+1}\Big)^{n} \cdot \sum_{i=0}^{m-1}{n-1\choose i}k^{n-i}+ \Big(\frac{|X|}{k+1}\Big)^{n} \cdot \sum_{i=1}^{m-1}{n-1\choose i-1}k^{n-i}\\
=&\Big(\frac{|X|}{k+1}\Big)^n\cdot \sum_{i=0}^{m-1}{n\choose i}k^{n-i}. \qedhere
	\end{align*}}
\end{proof}

Observe that one can assume without loss of generality that the sequence $\lambda_1,\,\dots\,,\lambda_m$ from the statement of Theorem~\ref{th chi grid} is non-decreasing. Thus, from Corollary~\ref{cor11} and \eqref{eqpighole} it easily follows that, if such a tessellatable $X$ exists, then, as $n \rightarrow \infty$, we have
\begin{equation*}
	\chi\big(\R_\infty^n, \B_{k}(\lambda_1)\times\dots\times \B_{k}(\lambda_m)\big) \ge  \ \frac{|X|^n}{\alpha\big(X^n, \B_{k}(\lambda_1)\times\dots\times \B_{k}(\lambda_m) \big)}   \ge \left(\frac{k +1}{k}+o(1)\right)^n.
\end{equation*}

Hence, to complete the proof of Theorem~\ref{th chi grid} it is enough to show that there is a finite set $X \subset \R$ that satisfies conditions of Corollary~\ref{cor11}.

\subsection{Tessellatable and almost tessellatable sets}

Let us consider some examples. Both $\B_1(2)$ and $\B_1(3)$ tessellates $X=[12]$, since
\begin{align*}
	[12] = & \ \{1,3\}\cup\{2,4\} \cup\{5,7\}\cup\{6,8\} \cup\{9,11\}\cup\{10,12\} \\
	= & \ \{1,4\}\cup\{2,5\} \cup\{3,6\}\cup\{7,10\} \cup\{8,11\}\cup\{9,12\}.
\end{align*}
More generally, if $\lambda_i \in \N$ for all $i \in [m]$ then 
\begin{equation*}
	X = \big[(k+1)\cdot\text{lcm}(\lambda_1,\,\dots\,, \lambda_m)\big],
\end{equation*}
is $\B_k(\lambda_i)$-tessellatable for all $i \in [m]$, where lcm stands for the least common multiple. 

Another simple example is the case when all $\lambda_i$ are linearly independent over $\Z$. In this case, each $\B_{k}(\lambda_i)$ tessellates their Minkowski sum
\begin{equation*}
	X=\B_{k}(\lambda_1)+\dots+\B_{k}(\lambda_m) =
	\{c_1\lambda_1+\dots+ c_m\lambda_m: c_i \in [k]_0 \mbox{ for all } i \in [m]\}.
\end{equation*}

Unfortunately, it is not clear how to construct such an $X$ in general, and we suspect that it may not always exist, see the discussion in~\cite{mathoverflow}. However, the argument from the proof of Lemma~\ref{lemma tessel} may work in a weaker setting.

Given $m \in \N$, $k_1,\,\dots\,,k_m \in \N$, positive  reals $\lambda_1\le\dots\le\lambda_m$, and a finite $X \subset \R$, assume that a subset $Y \subset X$ is $\B_{k_1}(\lambda_1)$-tessellatable. Running the same argument as in  Lemma~\ref{lemma tessel}, one can show that for all $n \in \N$ we have

\begin{align} \label{eq tessel}
	\alpha\big(X^n, \B_{k_1}(\lambda_1)\times\dots\times \B_{k_m}(\lambda_m) \big) \le \left(|X|-\frac{|Y|}{k_1+1}\right)\cdot& \, \alpha\big(X^{n-1}, \B_{k_1}(\lambda_1)\times\dots\times \B_{k_m}(\lambda_m) \big) \nonumber \\
	 + \left(\frac{|Y|}{k_1+1}\right) \cdot&\, \alpha\big(X^{n-1}, \B_{k_2}(\lambda_2)\times\dots\times \B_{k_m}(\lambda_m) \big).
\end{align}

Let $k \in \N$, $\lambda \in \R$, and $\varepsilon < 1$ be positive. We call a finite $X\subset\R$ {\it $\varepsilon$-almost $\B_k(\lambda)$-tessellatable} if and only if there is a $\B_k(\lambda)$-tessellatable subset $Y\subset X$ of cardinality at least $(1-\varepsilon)|X|$.

Given $k, m \in \N$, positive reals $\lambda_1 \leq \dots \leq \lambda_m$, and a small positive $\varepsilon < 1$, assume that there is a finite subset $X\subset\R$ such that $X$ is $\varepsilon$-almost $\B_{k}(\lambda_i)$-tessellatable for all $i \in [m]$. Then, by solving the recurrence relation \eqref{eq tessel}, one can get the following inequality that generalises Corollary~\ref{cor11}:
$$\alpha\big(X^n, \B_{k}(\lambda_1)\times\dots\times \B_{k}(\lambda_m) \big) \le \Big(\frac{|X|}{k+1}\Big)^n\cdot \sum_{i=0}^{m-1}{n\choose i}(k+\varepsilon)^{n-i}.$$
Therefore, as $n \rightarrow \infty$, we get that
\begin{equation*}
	\chi\big(\R_\infty^n, \B_{k}(\lambda_1)\times\dots\times \B_{k}(\lambda_m)\big) \ge  \ \frac{|X|^n}{\alpha\big(X^n, \B_{k}(\lambda_1)\times\dots\times \B_{k}(\lambda_m) \big)}   \ge \left(\frac{k+1}{k+\varepsilon}+o(1)\right)^n.
\end{equation*}
Clearly, if we could make $\varepsilon$ arbitrarily small, we would obtain the desired asymptotic lower bound
\begin{equation*}
	\chi\big(\R_\infty^n, \B_{k}(\lambda_1)\times\dots\times \B_{k}(\lambda_m)\big) \ge \left(\frac{k+1}{k}+o(1)\right)^n.
\end{equation*}
So, to finish the proof of Theorem~\ref{th chi grid} it remains only to prove the following lemma.

\begin{Lemma} \label{lemma exist almost tessel}
	Given $k, m \in \N$, positive  reals $\lambda_1,\,\dots\,,\lambda_m$ and a positive $\varepsilon < 1$, there is a finite subset $X\subset\R$ such that $X$ is $\varepsilon$-almost $\B_{k}(\lambda_i)$-tessellatable for all $i \in [m]$.
\end{Lemma}

\begin{proof}
	As in Section~\ref{sec4}, by $\langle \lam \rangle$ we denote an abelian group generated by $\lambda_1,\,\dots\,,\lambda_m$, i.e.,
	\begin{equation*}
		\langle \lam \rangle = \{a_1\lambda_1+\dots+a_m\lambda_m: a_1,\,\dots\,, a_m \in \Z\}.
	\end{equation*}
	Recall that the fundamental theorem of finitely generated abelian groups implies that $\langle \lam \rangle \cong \Z^t$ for some $t \in [m]$ since $\langle \lam \rangle$ is trivially a torsion-free group.  Let $\varphi: \langle \lam \rangle \rightarrow \Z^t$ be one such isomorphism.
	
	Fix $i \in [m]$ and set $S_i = \varphi\big(\B_{k}(\lambda_i)\big)$. Now we show that $S_i$ tiles $\Z^t$, i.e., that there is $A_i \subset \Z^t$ such that $\Z^t = \bigsqcup_{\x \in A_i} (S_i+\x),$ where $\sqcup$ denotes disjoint union. To do so, we first observe that the union $\bigsqcup_{0 \le x < \lambda_i} (\B_{k}(\lambda_i)+x)$ is disjoint and represents a right-open interval $\big[0; (k+1)\lambda_i\big)$. Thus,
	\begin{equation*}
		\R = \bigsqcup_{j\in \Z} \bigsqcup_{0\le x < \lambda_i}\big(\B_{k}(\lambda_i)+x+j(k+1)\lambda_i\big).
	\end{equation*}
	Since $\langle \lam \rangle$ is a group that contains $\lambda_i$, we can partition it into disjoint translates of $\B_k(\lambda_i)$ as well:
	\begin{equation*}
		\langle \lam \rangle = \bigsqcup_{j\in \Z} \bigsqcup_{\substack{ x \in \langle \lam \rangle \\ 0\le x < \lambda_i}}\big(\B_{k}(\lambda_i)+x+j(k+1)\lambda_i\big).
	\end{equation*}
	By applying the isomorphism $\varphi: \langle \lam \rangle \rightarrow \Z^t$ to both sides of the previous equality, we find the desired representation
	\begin{equation} \label{eq lemma exist almost tessel 1}
		\Z^t = \bigsqcup_{\x \in A_i} (S_i+\x)
	\end{equation}
	with 
	\begin{equation*}
		A_i = \big\{\varphi(x)+j(k+1)\varphi(\lambda_i): j \in \Z, x \in \langle \lam \rangle, 0 \le x < \lambda_i\big\}.
	\end{equation*}

	Given $r \in \N$, set
	\begin{equation*}
		A_{i,r} = \{\x \in A_i: S_i+\x \subset [r]^t\}, \ Y_{i,r}' = \bigsqcup_{\x \in A_{i,r}}(S_i+\x).
	\end{equation*}
	Clearly, $Y_{i,r}'$ is an $S_i$-tessellatable subset of $[r]^t$ by construction. Therefore, $Y_{i,r} \coloneqq \varphi^{-1}(Y_{i,r}')$ is a $\B_{k}(\lambda_i)$-tessellatable subset of $X_r \coloneqq \varphi^{-1}([r]^t) \subset \R$.
	
	Moreover, it follows from \eqref{eq lemma exist almost tessel 1} that there is a $c_i \in \N$ depending only on the diameter of $S_i$ such that for all $r \in \N$, the set $Y_{i,r}'$ contains the `central subcube' of $[r]^t$ with side length of $r-2c_i$, i.e., we have $[c_i+1,r-c_i]^t \subset Y_{i,r}'$. Hence,
	\begin{equation*}
		\frac{|Y_{i,r}|}{|X_r|} = \frac{|Y_{i,r}'|}{|[r]^t|} \ge \frac{(r-2c_i)^t}{r^t} = \left(1-\frac{2c_i}{r}\right)^t.
	\end{equation*}
	In particular, there exists $r_i\in \N$ such that for all $r\ge r_i$, we have $|Y_{i,r}| > (1-\varepsilon)|X_r|$. Thus, we conclude that $X_r$ is $\varepsilon$-almost $\B_k(\lambda_i)$-tessellatable for all $r \ge r_i$.
	
	Now to finish the proof it remains only to put $X \coloneqq X_{r'}$, where $r' = \max_{i}\{r_i\}.$
\end{proof}

\section{Infinite metric spaces}\label{sec6}

We split this section into four parts. In the first one we give the upper bound on $\chi(\R_\infty^n, \M)$ for an infinite metric space $\M$ in a simple special case when $\M$ is unbounded. In the second part we give the upper bound in case of a general (not necessarily unbounded) infinite $\M$ and prove Proposition~\ref{prop upper for infinite intro}. Finally, in the last two parts we give the lower bounds on $\chi(\R_\infty^n, \M)$ in a special case when $\M$ is a geometric progression $\mathcal{G}(q)$ and prove Theorems~\ref{cor geom baton intro} and~\ref{thminf2 intro} respectively.

\subsection{Upper bound in the unbounded case}

Recall that in the Euclidean case, for any  dimension $n$ and any infinite metric space $\mathcal M$, there exists a $2$-colouring of $\R_2^n$ with no monochromatic copy of $\mathcal M$.
In the max-norm case $2$ colours may sometimes be enough as well. The following simple statement is true for any norm and is rather folklore. However, it does not seem to explicitly appear in the literature, thus we include the proof for completeness.

\begin{Proposition} \label{prob unbounded}
	For all $n \in \N$ and for all unbounded metric spaces $\M$, there is a $2$-colouring of~$\R_{\infty}^n$ with no monochromatic copy of $\mathcal M$.
\end{Proposition}
\begin{proof}
	Let $\M=(M,\rho_M)$ be an unbounded metric space. We will find the desired colouring along a sequence of nested hypercubes with sufficiently quickly growing side lengths and alternating colours.
	
	Pick an element $m\in M$, and let $H_0=\emptyset$. We will recursively define a sequence of nested closed hypercubes $H_i = [-h_i,h_i]^n$ and elements $m_i \in M$ for all $i \in \N$ with the following properties: 
	\begin{enumerate}
		\item $\bigcup_i H_i=\mathbb{R}^n$, i.e. $h_i \rightarrow \infty$ as $i \rightarrow \infty$.
		\item For any $i\in \mathbb{N}$ and for any isometry $\varphi:\mathcal{M}\to \mathbb{R}^n_{\infty}$, if $\varphi(m)\in H_{i}\setminus H_{i-1}$, then $\varphi(m_i)\in H_{i+1}\setminus H_{i}$.
	\end{enumerate}
	By colouring each `ring' $H_i\!\setminus\! H_{i-1}$ with colour $i$ mod $2$ we clearly get a proper  $2$-colouring, since any embedding $\M$ into $\mathbb{R}_{\infty}^n$ intersects two consecutive rings.
	
	To find such sequences, in the first step choose $h_1=1$. Assume that we have already defined $m_1, \,\dots\, ,m_{i-1}$ and $h_1,\,\dots\,,h_{i}$. Then in the $i'$th step pick $m_{i} \in M$ so that
	$$\rho_M(m,m_i) > \text{diam}(H_i)=2h_i$$
	(this is possible, since $\mathcal{M}$ is not bounded) and put
	$$h_{i+1}=h_{i}+\rho_M(m,m_i).$$
	
	We clearly have $h_{i+1}> 3h_i$ by construction, and thus  $h_i\rightarrow \infty$ as $i \rightarrow \infty$. To check the second required property fix an isometry $\varphi:\mathcal{M}\to \mathbb{R}^n_{\infty}$. Assume that $\varphi(m)\in H_{i}\!\setminus\! H_{i-1}$. Observe that the choice of $m_i$ ensures that $\varphi(m_i) \not\in H_i$, while the size of $h_{i+1}$ guarantees that $\varphi(m_i) \in H_{i+1}$, which completes the proof.
\end{proof}

\subsection{General upper bound: proof of Proposition~\ref{prop upper for infinite intro}}

Now we turn to the proof of the general inequality $\chi(\R_\infty^n, \M) \le n+1$, that holds for all $n \in \N$ and for an arbitrary (not necessarily unbounded) infinite $\M \subset \R_{\infty}^n$.

We will need the following well-known application of the Lov\'asz Local Lemma. The first version of this result appeared in a paper \cite{ErLo1975} of Erd\H{o}s and Lov\'asz himself, while its later improvements are due to Alon--K\v{r}\'\i\v{z}--Ne\v{s}et\v{r}il \cite{AlKrNes} and Harris--Srinivasan \cite{HarSrin}, see also Axenovich et al. \cite{Axe}.

\begin{Theorem}[\cite{HarSrin}] \label{poly}
	For all $n \in \N$, there exists $k=k(n) \in \N$ such that for every subset $S \subset \R$ of size $k+1$ the following holds. There exists a partition $\R = \bigsqcup_{i=1}^{n+1}C_i$ of $\R$ into $n+1$ disjoint parts such that for all $x \in \R$ both sets $S+x$ and $-S+x$ intersect each part.  Moreover, one may assume that $k(n) = (1+o(1))n\ln n$ as $n \rightarrow \infty$.
\end{Theorem}

Given $n \in \N$ and an infinite $\M \subset \R_{\infty}^n$, let us pick an arbitrary set of $k(n)^n+1$ of its points, where $k(n)$ is from the previous theorem. It follows from Proposition~\ref{prop no batons} that this finite set contains an isometric copy of some $(k(n)+1)$-point baton $\B$. It is clear that $\chi(\R_\infty^n, \M) \le \chi(\R_\infty^n, \B)$. Hence, to finish the proof of Proposition~\ref{prop upper for infinite intro}, it is sufficient to prove the following statement.

\begin{Proposition} \label{prop upper for long baton}
	Let $n \in \N$ and $\B \subset \R$ be an arbitrary baton of cardinality $k(n)+1$, where $k(n)$ is from Theorem~\ref{poly}. Then we have $\chi(\R_{\infty}^n, \B) \le n+1$.
\end{Proposition}
\begin{proof}
	Apply Theorem~\ref{poly} with $\B$ playing the role of $S$ to find a specific partition $\R = \bigsqcup_{i=1}^{n+1}C_i$ of $\R$ into $n+1$ disjoint parts. Given $i \in [n+1]$, for all $x \in \R$, both sets $\B+x$ and $-\B+x$ intersects $C_i$ by construction. Thus, $A_i \coloneqq \R\!\setminus\!C_i \subset \R$ is $\B$-tr-free. Then $A_i^n \subset \R^n$ is $\B$-free by Corollary~\ref{cor power free}. So, we can assign an individual colour to each $A_i^n, i \in [n+1]$.
	
	To finish the proof, it remains only to observe that $\bigcup_{i=1}^{n+1}A_i^n = \R^n$. Indeed, assume the contrary, that is there exists $\x = (x_1,\,\dots\,,x_n) \in \R^n$ that does not belong to this union. Thus, for all $i \in [n+1]$, there is $j_i \in [n]$ such that $x_{j_i} \notin A_i$, i.e., that $x_{j_i} \in C_i$. By the pigeonhole principle, there exist $j \in [n]$ and distinct $i_1, i_2 \in [n+1]$ such that $x_j \in C_{i_1} \cap C_{i_2}$. Since $C_{i_1}$ and $C_{i_2}$ are disjoint by construction, this is a contradiction.
\end{proof}

\subsection{Proof of Theorem~\ref{cor geom baton intro}}

In the present subsection we show that the upper bound from Proposition~\ref{prop upper for infinite intro} is sometimes tight, i.e., that for all $n \in \N$ there is an infinite $\M \subset \R_{\infty}^n$ such that the chromatic number $\chi(\R_\infty^n, \M)$ equals $n+1$ exactly.

First, recall from the introduction that for $q>0$ we defined an infinite set of reals
\[\g(q) \coloneqq \{0,1,1+q,1+q+q^2,1+q+q^2+q^3,\,\dots\}\]
equipped with the standard metric on $\R$.
When the choice of $q$ is clear from the context, we simply write $\g$. For $i\in \mathbb Z_{\ge 0}$ and $q < 1$, let us put $t_i \coloneqq \frac {q^i}{1-q}$, and $T \coloneqq \{t_0,t_1,t_2,\,\ldots\}.$

We introduce some further related notions. We say that a set $\mathcal{C} \subset \R_{\infty}^n$ is a {\it proper isometric copy} of $\g$ if $\mathcal{C}$ is isometric copy of $\g$ and, moreover, there is $i \in [n]$ such that the $i$-th coordinates of the points of $\mathcal{C}$ form a translation of $-\g$ (in particular, the $i$-th coordinate of the limit point of $\mathcal{C}$ is smaller than all the $i$-th coordinates of the points of $\mathcal{C}$). In this case, we will also say that this proper isometric copy $\mathcal{C}$ has \emph{direction} $i$. For example, the set $T = \frac{1}{1-q}-\g \subset \R$ is a proper isometric copy of $\g$, while the set $\g\subset\R$ itself is not.

Let us recall the statement of Theorem \ref{cor geom baton intro} from the introduction. It states that for all $n \in \N$, there exists $q_0>0$ such that if $0<q\le q_0$, then $\chi(\R_\infty^n, \g(q)) = n+1.$ Taking into account the upper bound from Proposition~\ref{prop upper for infinite intro}, it remains only to get the lower one. Here we prove an even stronger statement.

\begin{Theorem} \label{thminf1}
	For any $n\in \mathbb N$ and $0<\varepsilon \le \frac{1}{2}$, there exists $q_0>0$ such that the following holds for any positive $q\le q_0$. In any $n$-colouring of the cube $\mathcal{H}=[-\varepsilon,1+\varepsilon]^n$ with the maximum metric we can find a monochromatic proper isometric copy of
	$\g(q).$
\end{Theorem}
\begin{proof}
	We prove the statement with $q_0(n,\varepsilon) = \frac{\varepsilon}{2^{n+1}}$ by induction on $n.$ The case $n=1$ is trivial, since $\text{diam}(\g) = \frac{1}{1-q} <1+2\varepsilon$. So, we turn to the induction step.
	
	Set $\sig=(1,\,\dots\,,1) \in \R^n$ and consider the colouring of $\mathcal H$ with colours from $[n]$. Without loss of generality, assume that the point $\x^0 \coloneqq (t_0,0,\,\ldots\,,0)+ \frac{\varepsilon}{2}\sig$ is coloured with colour $n$. (One can easily check that $\x^0 \in \mathcal H$ since $q\le \frac{\varepsilon}{8}$.) Let us consider the sequence of $(n-1)$-dimensional cubes  $I_i\subset \mathcal H,$ defined  for all $i \in \N$ as $I_i \coloneqq \{t_i\}\times [0, t_{i+1}]^{n-1}+\frac{\varepsilon}{2}\sig$, where the plus sign stands for the Minkowski sum. Assume that for every $i\in \mathbb N$ there is a point $\x^i\in I_i$ of colour $n$. It is easy to see that $|x^i_1-x^{i+1}_1| = t_i-t_{i+1} = q^i$, while for all $j \neq 1$ we have $|x^i_j-x^{i+1}_j| \leq t_{i+1}= \frac{q^{i+1}}{1-q}<q^i$, provided by the inequality $q \le \frac{\varepsilon}{8} < \frac{1}{8}$. Thus, $\|\x^i-\x^{i+1}\|_{\infty} = q^i$ for all $i \in \N$, and the set $\{\x^0,\x^1,\x^2,\,\ldots \}$ is a proper isometric copy of $\mathcal{G}$ having direction $1$ (see Lemma~\ref{lemma baton embed}). Moreover, it is entirely coloured with colour $n$ by construction, a contradiction. Therefore, we may assume that there is an $i \ge 1$ such that $I_i$ is entirely coloured with colours from $[n-1]$. 

	Let $j$ be the smallest positive integer such that there exists an axis parallel $(n-1)$-dimensional cube $H$ of side length $(1+2\varepsilon) q^{j}$ inside $\big[\frac{-\varepsilon}{2j},\frac\varepsilon{2j}\big]^n+ \frac{\varepsilon}{2}\sig \subset \mathcal H$ that is entirely coloured with some $n-1$ colours from $[n]$. Note that the cube $I_i$ satisfies $I_i\subset \big[\frac{-\varepsilon}{2(i+2)}, \frac\varepsilon{2(i+2)}\big]^n+ \frac{\varepsilon}{2}\sig$ (since $t_i < \frac{\varepsilon}{2(i+2)}$ whenever $q\leq \frac{\varepsilon}{8}$) and its side length $t_{i+1}$ is greater that $(1+2\varepsilon)q^{i+2}$ (since $\varepsilon < 1)$. Thus, such $j$ exists and does not exceed $i+2$.
	
	Now we show that there exists $\mathcal C_1\subset [0,\varepsilon]^n$, that is a monochromatic proper isometric copy of  $q\cdot \mathcal{G}$. Observe that if $j=1$, then $H \subset [\frac{-\varepsilon}{2},\frac{\varepsilon}{2}]^n+ \frac{\varepsilon}{2}\sig = [0,\varepsilon]^n$. Hence, the existence of such $\mathcal{C}_1$ holds by induction, applied to $H$ and the set $q\cdot \mathcal{G}$. So, let us assume that $j \ge 2$. By induction, applied to $H$ and the set $q^{j}\cdot \mathcal{G}$, we can find a monochromatic proper isometric copy $\mathcal C_j$ of $q^{j}\cdot \mathcal{G}$ inside $H$. Let $\mathcal C_j=\{\y^j,\y^{j+1},\, \ldots\}$. Consider the $(n-1)$-dimensional cube $H_{j-1}$ of points that are at distance $\sum_{s=j-1}^{m-1} q^{s}$ apart from the point $\y^{m}$, for any $m \ge j.$ In other words, if we assume that $\mathcal C_j$ has direction $1$, then
	\[H_{j-1}=\y^j+\{(x_1,\,\ldots\,, x_n)\in \R^n: x_1 = q^{j-1},\  -q^{j-1}\le x_i\le q^{j-1}\text{ for } 2\le i\le n\}.\]
	It is easy to check that $2q^{j-1}$, the side length of $H_{j-1}$, is not less than $(1+2\varepsilon)q^{j-1}$ (as far as $\varepsilon \le \frac{1}{2}$) and that
	$$H_{j-1}\subset \Big[\frac{-\varepsilon}{2(j-1)},\frac\varepsilon{2(j-1)}\Big]^n+\frac{\varepsilon}{2}\sig$$
	since $\frac{\varepsilon}{2j}+q^{j-1} \le \frac{\varepsilon}{2(j-1)}$ whenever $q \le \frac{\varepsilon}{8}$. Hence, by the minimality of the choice of $j$, the set $H_{j-1}$ has points of all colours. In particular, there is a point $\y^{j-1}\in H_{j-1}$ that is of the same colour as $\mathcal C_{j}$. 
	Define $\mathcal C_{j-1} \coloneqq \{\y^{j-1},\y^{j},\,\dots\}$ and note that $\mathcal{C}_{j-1}$ is a proper isometric copy of $q^{j-1}\cdot\g$. Repeat the same argument with $\mathcal C_{j-1}$ playing the role of $\mathcal C_j,$ obtaining a point $\y^{j-2}$ of the same colour contained in the $(n-1)$-dimensional cube $H_{j-2}\subset \big[\frac{-\varepsilon}{2(j-2)},\frac\varepsilon{2(j-2)}\big]^n+\frac{\varepsilon}{2}\sig$ of side length no less than $(1+2\varepsilon)q^{j-2}$, and put $\mathcal C_{j-2} \coloneqq \{\y^{j-2},\y^{j-1},\,\dots\}$. Run the same argument recursively until we get the desired $\mathcal C_1$, a monochromatic proper isometric copy of  $q\cdot \mathcal{G}$ inside $[\frac{-\varepsilon}{2},\frac{\varepsilon}{2}]^n+ \frac{\varepsilon}{2}\sig = [0,\varepsilon]^n$.
	
	Now we define $H_0$ in the same manner. Again, without loss of generality, we assume that $\mathcal C_1 = \{\y^1,\y^2,\,\dots\}$ has direction $1$. Then we put $H_0 \coloneqq \{1+y^1_1\}\times[-\varepsilon,1]^{n-1} \subset \mathcal{H}$. A priori, there are two possibilities. First, $H_0$ may contain a point $\y^0$ of the same colour as that of $\mathcal C_1$. In this case, it is not hard to check that $\{\y^0,\y^1,\,\dots\}$ is a monochromatic proper isometric copy of $\g$ since $\|\y^0-\y^1\|_{\infty} = y^0_1-y^1_1=1$ (again, see Lemma~\ref{lemma baton embed}). Second, $H_0$ may have no points of the colour of $\mathcal C_1$. In this case, since $H_0$ is an $(n-1)$-dimensional cube of side length $1+\varepsilon$, we can apply induction\footnote{We can indeed induct in such a way because $q_0(n,\varepsilon) = \frac{\varepsilon}{2^{n+1}} = \frac{\varepsilon/2}{2^n}= q_0(n-1,\varepsilon/2)$. Observe that this is the only place in the proof where we can not make $q_0$ independent of $n$.} with $n-1$ playing the role of $n$ and $\varepsilon/2$ playing the role of $\varepsilon$ to get a monochromatic proper isometric copy of $\g$ inside $H_0$. In any case, this concludes the proof.
\end{proof}

\subsection{Proof of Theorem~\ref{thminf2 intro}}

Finally, we show the existence of a {\it single} positive value of $q_0$ such that for any $0<q\le q_0$, the chromatic number $\chi(\R_{\infty}^n, \g(q))$ grows at last as fast as $\log_{3}n$ is. Since $\chi(\R_{\infty}^n, \g(q))$ is clearly non-decreasing as a function on $n$, we may assume without loss of generality that $n=3^k$ for some $k \in \N$. As in the previous subsection, here we prove an even stronger statement, that deals with a `small' $n$-dimensional cube only instead of dealing with the whole space $\R_\infty^n$.

\begin{Theorem}\label{thminf2}
	For any $0<\varepsilon<\frac{1}{2}$ there exists $q_0>0$ such that the following holds for any positive $q\le q_0$. If $n= 3^k$, then in any $k$-colouring of the cube $\mathcal{H} \coloneqq [-\varepsilon,1+\varepsilon]^n$ with the maximum metric one can always find a monochromatic proper isometric copy of $\mathcal{G}(q).$
\end{Theorem}

We need some preparation before the proof. Given a vertex $\vv$ of $\mathcal{H}$, we associate it with its `signum vector' $\sig^{\vv} = (\sigma^{\vv}_1,\,\dots\,,\sigma^{\vv}_n) \in \R^n,$ where $\sigma^{\vv}_i = 1$ if $v_i=-\varepsilon$ and $\sigma^{\vv}_i = -1$ if $v_i=1+\varepsilon$, $i \in [n]$. There is also a natural bijection between the vertices of $\mathcal{H}$ and the subset of $[n]$: we associate the vertex $\vv$ with the subset $S(\vv)=\{i: v_i=1+\varepsilon\}$. We will say that a set $A\subseteq \mathbb{R}^n$ is {\it $\varepsilon$-close to a vertex $\vv$} of the cube $\mathcal H=[-\varepsilon,1+\varepsilon]^n$ if any point $\mathbf{w}\in A$ satisfies $\|\mathbf{w}-\vv\|_\infty\le 2\varepsilon$. Note that if $\varepsilon <\frac{1}{2}$, then any set can be close to at most one vertex of $\mathcal H$.

Finally, we need some notions and theorems related to the VC-dimension of set systems. We say that $\mathcal{F}\subseteq 2^{[n]}$ \emph{shatters} $X\subseteq [n]$ if $\{F\cap X: F \in \mathcal{F}\}=2^{X}$. The {\it VC-dimension} of $\mathcal{F}$ is the size of the largest set shattered by $\mathcal{F}$. A well-known result of Sauer \cite{Sauer} , Shelah \cite{Shelah} and Vapnik--Chervonenkis \cite{Vapnik} bounds the cardinality of a family in terms of its VC-dimension.

\begin{Theorem}[Sauer, Shelah, Vapnik--Chervonenkis]
	If $\mathcal{F}\subseteq 2^{[n]}$ has VC-dimension at most $d$, then
	\[|\mathcal{F}|\leq \sum_{i=0}^d \binom{n}{i}.\]
\end{Theorem}

As a direct corollary, we obtain that if $|\mathcal{F}|$ is large, then $\mathcal{F}$ shatters a large set.

\begin{Corollary}\label{VC} If $|\mathcal{F}|> \sum_{i=0}^d\binom{n}{i}$, then there is a subset $X\subseteq [n]$ of size $d+1$ that is shattered by $\mathcal{F}$. In particular, if $k \ge 4$, $n=3^k$, and $|\mathcal{F}|\geq \frac{2^{n-1}}{kn}$, then $\mathcal{F}$ shatters a set of size $\frac{n}{3}$.
\end{Corollary}

\begin{proof}[Proof of Theorem~\ref{thminf2}]
	We will prove the following stronger statement by induction on $k$:
	
	\medskip
	
	If $n=3^k$, then in any $k$-colouring of the cube $\mathcal H=[-\varepsilon,1+\varepsilon]^n$ with the maximum metric one can always find a monochromatic proper isometric copy $\{\x^0,\x^1,\,\ldots\}$ of $\mathcal{G}$ such that $\{\x^1,\x^2,\,\ldots\}$ is $\varepsilon$-close to some vertex of the cube $\mathcal{H}$.
	
	\medskip
	
	To prove that base of induction, observe that the additional requirement on the $\varepsilon$-closeness property is satisfied by the monochromatic proper isometric copies of $\mathcal{G}$ that we found in the proof of Theorem~\ref{thminf1}. Thus, the statement is true for any $k\le 3$ with $q_0$ as is guaranteed in that theorem, namely, $q_0 = \frac{\varepsilon}{2^{3+1}} = \frac{\varepsilon}{16}$. So, from now on we may assume that $k \ge 4$. We may also suppose that $[k]$ is the set of colours.
	
	The proof of the induction step builds on the proof of Theorem~\ref{thminf1}. First, we will run a procedure similar to the one in the previous proof starting from a vertex $\vv$ of $\mathcal{H}$. Namely, we will find an $\varepsilon$-close to $\vv$ monochromatic proper isometric copy $\mathcal{C}_1=\mathcal{C}_1^{\vv}=\{\y^1,\y^2,\,\dots\}$ of $q\cdot \mathcal{G}$. Despite the fact that this part is almost identical to the one from the previous proof, it still will be helpful to spell it out in this more general form. Second, we will use these copies of $q\cdot \mathcal{G}$ to find the desired monochromatic proper isometric copy of $\g$ with the additional requirement on the $\varepsilon$-closeness property. Note that this step was quite simple in the proof of Theorem~\ref{thminf1}. However, this will be the hard part of the current proof.

Fix a vertex $\vv$ of $\mathcal H$ such that $|S(\vv)| < \frac{n}{2}$. We may assume without loss of generality that $1\notin S(\vv)$, that is $v_1= -\varepsilon$. Put $\x^0 \coloneqq \vv+(t_0,0,\,\dots\,,0)+ \frac{3\varepsilon}{2}\sig^{\vv}$. Further, let $c=c(\x^0)$ be the colour of $\x^0$. We consider the sequence of $(n-1)$-dimensional cubes  $I_i\subset \mathcal H,$ defined  for all $i \in \N$ as
\begin{equation*}
	I_i \coloneqq \vv+ \{t_i\}\times [0, t_{i+1}]^{n-1}+\frac{3\varepsilon}{2}\sig^{\vv}.
\end{equation*}
Assume that for every $i\in \mathbb N$ there is a point $\x^i\in I_i$ of colour $c$. Then the set $\{\x^0,\x^1,\x^2,\,\ldots \}$ is a monochromatic proper isometric copy of $\mathcal{G}$ such that $\{\x^1,\x^2,\,\dots\}$ is $\varepsilon$-close to $\vv$. (Where this last property holds since $q\leq \frac{\varepsilon}{8}$.) Therefore, we may assume that there is an $i\ge 1$ such that the $(n-1)$-dimensional cube $I_i$ is entirely coloured with colours from $[k]\!\setminus\! \{c\}$.

Let $j$ be the smallest positive integer such that there exists an axis parallel $(n-1)$-dimensional cube $H'$ of side length $(1+2\varepsilon) q^{j}$ inside $\vv+\big[\frac{-\varepsilon}{2j},\frac\varepsilon{2j}\big]^n+ \frac{3\varepsilon}{2}\sig^{\vv} \subset \mathcal H$ that is entirely coloured with some $k-1$ colours from $[k]$. As in the proof of Theorem~\ref{thminf1}, observe that the cube $I_i$ satisfies $I_i\subset \vv+ \big[\frac{-\varepsilon}{2(i+2)}, \frac\varepsilon{2(i+2)}\big]^n+ \frac{3\varepsilon}{2}\sig^{\vv}$, has side length $t_{i+1}$ being greater that $(1+2\varepsilon)q^{i+2}$, and entirely coloured with colours from $[k]\!\setminus\! \{c\}$. Thus, such $j$ exists and does not exceed $i+2$.

Let $H$ be an $\frac{n}{3}$-dimensional face of $H'$ orthogonal to all the basic vectors $e_{r}$ such that $r \in S(\vv)$. Note that the existence of this face follows from the inequality $n-1-|S(\vv)| > \frac{n}{3}$ since $|S(\vv)| < \frac{n}{2}$.

Now we show that there exists a monochromatic proper isometric copy $\mathcal C_1=\mathcal{C}^{\vv}_1$ of  $q\cdot \mathcal{G}$ inside $\vv+ \big[\frac{-\varepsilon}{2}, \frac{\varepsilon}{2}\big]^n+ \frac{3\varepsilon}{2}\sig^{\vv}$. Observe that if $j=1$, then the existence of such $\mathcal{C}_1$ holds by induction applied with $k-1$ playing the role of $k$ to $H$ and the set $q\cdot \mathcal{G}$. (The induction indeed can be applied since $\dim H = 3^{k-1}$.) Note that for the direction $\ell=\ell(\vv)$ of this $\mathcal{C}_1$ we clearly have $\ell \notin S(\vv)$ by construction of $H$. So, let us assume that $j \ge 2$. By induction applied with $k-1$ playing the role of $k$ to $H$ and the set $q^j\cdot \mathcal{G}$, we can find a monochromatic proper isometric copy $\mathcal C_j = \{\y^j,\y^{j+1},\, \ldots\}$ of $q^{j}\cdot \mathcal{G}$ inside $H$. Consider the $(n-1)$-dimensional cube $H_{j-1}$ of points that are at distance $\sum_{s=j-1}^{m-1} q^{s}$ apart from the point $\y^{m}$, for any $m \ge j.$ In other words, if we assume that $\mathcal C_j$ has direction $\ell=\ell(\vv)$ (again, we have $\ell \notin S(\vv)$ by construction), then
\[H_{j-1}=\y^j+\big\{(x_1,\,\ldots\,, x_n)\in \R^n: x_\ell = q^{j-1},\  -q^{j-1}\le x_i\le q^{j-1}\text{ for } i\in [n]\!\setminus\!\{\ell\}\big\}.\]
As in the proof if Theorem~\ref{thminf1}, it is easy to check that $2q^{j-1}$, the side length of $H_{j-1}$, is not less than $(1+2\varepsilon)q^{j-1}$ and that
$$H_{j-1}\subset \vv+ \Big[\frac{-\varepsilon}{2(j-1)},\frac\varepsilon{2(j-1)}\Big]^n+ \frac{3\varepsilon}{2}\sig^{\vv}.$$
Hence, by the minimality of the choice of $j$, the set $H_{j-1}$ has points of all colours. In particular, there is $\y^{j-1}\in H_{j-1}$ that is of the same colour as $\mathcal C_{j}$. 
Define $\mathcal C_{j-1} \coloneqq \{\y^{j-1},\y^{j},\,\dots\}$ and note that $\mathcal{C}_{j-1}$ is a monochromatic proper isometric copy of $q^{j-1}\cdot\g$ which has the same direction $\ell$ as $\mathcal{C}_j$. Repeat the same argument with $\mathcal C_{j-1}$ playing the role of $\mathcal C_j,$ obtaining a point $\y^{j-2}$ of the same colour contained in the $(n-1)$-dimensional cube $H_{j-2}\subset \vv+ \big[\frac{-\varepsilon}{2(j-2)},\frac{\varepsilon}{2(j-2)}\big]^n+ \frac{3\varepsilon}{2}\sig^{\vv}$ of side length no less than $(1+2\varepsilon)q^{j-2}$, and put $\mathcal C_{j-2} \coloneqq \{\y^{j-2},\y^{j-1},\,\dots\}$. Run the same argument recursively until we get the desired $\mathcal C_1$, a monochromatic proper isometric copy of  $q\cdot \mathcal{G}$ inside $\vv+ \big[\frac{-\varepsilon}{2}, \frac{\varepsilon}{2}\big]^n+ \frac{3\varepsilon}{2}\sig^{\vv}$.

Observe that $\mathcal C_1 = \{\y^1,\y^2,\,\dots\}$ is $\varepsilon$-close to $\vv$ by construction. Moreover, recall that its direction $\ell(\vv)$ does not belong to $S(\vv)$. In particular, this implies that $0 \le y^1_{\ell(\vv)}\le \varepsilon$. Thus, an axis parallel $(n-1)$-dimensional cube $H_0 = H_0^{\vv}$ of side length $1+\varepsilon$ defined by
\begin{align*}
	H_0^{\vv} \coloneqq \Big\{(x_1,\,\dots\,,x_n) \in \R^n: x_{\ell(\vv)}=1+y^1_{\ell(\vv)}, &\ 0 \le x_i \le 1+\varepsilon \mbox{ for } i \in S(\vv), \\
	 &-\varepsilon \le x_i \le 1 \mbox{ for } i \in [n]\!\setminus\! S(\vv)\! \setminus\! \{\ell(\vv)\}\Big\}
\end{align*}
is a subset of $\mathcal{H}$. If $H_0$ contains a point $\y^0$ of the same colour as that of $\mathcal C_1$, then the set $\{\y^0,\y^1,\,\dots\}$ is the desired monochromatic proper isometric copy of $\g$ such that $\{\y^1,\y^2,\,\dots\}$ is $\varepsilon$-close to $\vv$, and we are done. Indeed, this follows from the fact that $\|\y^0-\y^1\|_{\infty} = y^0_{\ell(\vv)}-y^1_{\ell(\vv)}=1$ (see Lemma~\ref{lemma baton embed}). Thus, we may assume that $H_0$ contains no point of the colour of $\mathcal C_1.$

Run the same procedure from any of $2^{n-1}$ vertices $\vv$ of $\mathcal{H}$ such that $|S(\vv)|< \frac{n}{2}$. This gives us a set of axis parallel $(n-1)$-dimensional cubes $H_0^{\vv}$ each of which is coloured with some $k-1$ colours from $[k]$. By the pigeonhole principle, we can choose a set $L$ of at least $\frac{2^{n-1}}{kn}$ of these vertices such that the following two statements hold. First, there is a direction, say $n$, such that $\ell(\vv)=n$ for any $\vv \in L$. Second, there is a colour, say $k$, such that $H_0^{\vv}$ is coloured with colours from $[k]\!\setminus\!\{k\} = [k-1]$ for any $\vv \in L$.

Let $S(L)=\{S(\vv): \vv\in L\} \subset 2^{[n]}$. Corollary~\ref{VC} implies that $S(L)$ shatters a subset $X\subset [n]$ of size $\frac{n}{3}$. Without loss of generality we may assume that $X = \{1,\,\ldots\,, \frac{n}{3}\}$. Choose a subset $L' \subset L$ of size $2^{n/3}$ such that for any $U \subseteq X$ there is a unique $\vv \in L'$ such that $S(\vv)\cap X = U$.

Consider the $\frac{n}{3}$-dimensional cube $Q$ of side length $1+2\varepsilon$, defined by
\begin{equation*}
	Q \coloneqq \Big\{(x_1,\,\dots\,,x_n) \in \R^n: -\varepsilon \le x_i \le 1+\varepsilon \mbox{ for } 1 \le i \le \frac{n}{3}, x_i = 0 \mbox{ for }  \frac{n}{3} < i < n, \ x_{n}=1+\varepsilon \Big\}.
\end{equation*}
Note that there is a natural bijection between the vertices of $Q$ and the subsets of $X$, similar to the one between the vertices of $\mathcal{H}$ and the subsets of $[n]$. Partition $Q$ into $2^{n/3}$ subcubes using hyperplanes $x_i=\frac{1}{2}$ for each $i\in X$, obtaining the partition
\begin{equation*}
	Q = \bigcup_{U\subseteq X} Q_U,
\end{equation*}
where $Q_U$ contains the vertex of $Q$ that corresponds to the set $U$. (We distribute the boundary parts arbitrarily.)

Let $F$ be a facet of $\mathcal H$ defined by the equation $x_n=1+\varepsilon$. Clearly, $Q$ is a subset of $F$. Moreover, given $\vv \in L'$, recall that $\ell(\vv)=n$. Therefore, $H_0^{\vv}$ is parallel to $F$ and the distance between them does not exceed $\varepsilon$ by construction. Thus, for all $\vv',\vv'' \in L'$, the distance between $H_0^{\vv'}$ and $H_0^{\vv''}$ does not exceed $\varepsilon$ as well.

Let $\pi: \mathcal{H}\to F$ be the orthogonal projection on the face $F$. Observe that for any $U \subseteq X$, we have $Q_U \subset \pi(H_0^{\vv})$, where $\vv$ is the unique element of $L'$ such that $S(\vv) \cap X = U$. Indeed, fix $\z \in Q_U$. Since $\pi(H_0^{\vv})$ is defined by the set of linear inequalities, it is sufficient to check that $\z$ satisfies them all to conclude that $\z \in \pi(H_0^{\vv})$. If $i \in U \subset S(\vv)$, then we clearly have $\frac{1}{2} \le z_i \le 1+\varepsilon$, which is stronger than the inequalities $0\le z_i \le 1+\varepsilon$ required by $\pi(H_0^{\vv})$. Similarly, if $i \in X\!\setminus\!U = X\!\setminus\!S(\vv)$, then $-\varepsilon \le z_i \le \frac{1}{2}$, which is again stronger than the inequalities $-\varepsilon\le z_i \le 1$ required by $\pi(H_0^{\vv})$. Further, if $\frac{n}{3}<i<n$, then $z_i=0$, and so there are no problems with these indexes as well. Finally, since both $Q_U$ and $\pi(H_0^{\vv})$ are subsets of $F$, there is nothing to check for $i=n$. Hence, we conclude that $\z \in \pi(H_0^{\vv})$.

Now we define a (new) colouring of $Q$, which is induced by the given colourings of the $H_0^{\vv}$'s as follows. Given $\z\in Q_U$, we find a unique $\vv \in L'$ such that $S(\vv)\cap X = U$, and a unique $\y\in H_0^{\vv}$ for which  $\pi(\y)=\z$. (Let us stress here once more that $\pi$ is the same for every $\vv$ and $U$, and that we do not orthogonally project on $Q_U$ itself, but only on the facet $F$.) Then we colour $\z$ with the colour of $\y$ and take the union of these colourings over all $U \subseteq X$.

Note that the side length of $Q$ equals $1+2\varepsilon$ and that $\dim Q = 3^{k-1}$. The new colouring of $Q$ constructed above uses colours only from $[k-1]$ (since we assumed that $H_0^{\vv}$ does not have a point of colour $k$ for any $\vv \in L'$). Thus, we can apply the induction with $k-1$ playing the role of $k$ to find a monochromatic proper isometric copy $\mathcal{C}_0 = \{\z^0, \z^1,\,\ldots\}$ of $\mathcal{G}(q)$ in $Q$ such that $\mathcal{C}_1 = \{\z^1,\z^2,\,\ldots\}$ is $\varepsilon$-close to some vertex $\w'$ of $Q$.

Using $\mathcal{C}_0$, we will find a monochromatic proper isometric copy of $\mathcal{G}(q)$ in the original colouring. First, consider the subset $U' \subseteq X$ that corresponds to $\w'$ and let $\vv'\in L'$ be the unique vertex such that $S(\vv')\cap X=U'$. Since $\mathcal{C}_1$ is $\varepsilon$-close to $\w'$, it is not hard to see that $\mathcal{C}_1$ is contained in $Q_{U'}$. Thus, there is a monochromatic set $\{\y^1,\y^2 ,\,\ldots\}\subset H_0^{\vv'},$ such that $\pi(\y^i) = \z^i$ holds for every $i$. Further, assume that $\z^0\in Q_{U''}$ for some other $U'' \subseteq X.$ Let $\vv''\in L'$ be the unique vertex such that $S(\vv'')\cap X = U''$ and let $\y^0\in H_0^{\vv''}$ be the point such that $\pi(\y^0) = \z^0.$ Recall that the hyperplanes $H_0^{\vv'}$ and $H_0^{\vv''}$ are parallel and the distance between them does not exceed $\varepsilon$. Thus, it is easy to see that the projection $\pi$ does not change the distances between $\y^i$'s. Thus, $\{\y^0,\y^1 ,\,\ldots\}$ is a monochromatic proper isometric copy of $\mathcal{G}(q)$ in the original colouring which has the same colour and the same direction as $\mathcal{C}_0$.

The final thing to remark is that $\{\y^1,\y^2 ,\,\ldots\}$ is $\varepsilon$-close to some vertex of $\mathcal H$, namely, the vertex that corresponds to the subset $U'\cup\{n\} \subset [n]$. This completes the proof.
\end{proof}

\section{Concluding remarks}\label{sec7}

Although we found the base of the exponent  $\chi_\M$ such that $\chi(\R_{\infty}^n,\M) =(\chi_\M+o(1))^n$ as $n \rightarrow \infty$ for all batons and for Cartesian products of arithmetic progressions, there is still a polynomial gap between the lower and upper bounds on  the  chromatic number. In particular, given $k,m \in \N$, Theorem~\ref{Chi Bkm} implies that $c_1n^{-m+1} \le \chi(\R_{\infty}^n,\B_k^m)\cdot\left(\frac{k+1}{k}\right)^{-n} \le c_2n$, where the values of $c_1=c_1(k,m)$ and $c_2=c_2(k,m)$ are independent of $n$. For a `rectangle' $\B(1)\times\B(2)$, one can easily deduce from the proof of Theorem~\ref{th chi grid} that $\frac{2}{n+2}\le \chi(\R_{\infty}^n,\B(1)\times\B(2))\cdot2^{-n} \le 1$. It would be interesting to get a polynomial improvement of some of these bounds.

For most other metric spaces $\M$ we do not even know the exact value of $\chi_\M$. One of the simplest examples is the `triangle' $\mathcal{T}(a,b,c)$ (i.e., a three-point metric space) with side lengths being $a, b$ and $c$. One may use an embedding $\mathcal{T}(2,2,3) \subset \B(1,2)\times\B(2)$ to get a lower bound for $\mathcal{T}(2,2,3)$. A positive answer to Problem~\ref{prb111} would imply that $\chi_{\mathcal{T}(2,2,3)} \ge 5/3$, while the  upper bound $\chi_{\mathcal{T}(2,2,3)} \le 5/3$ is immediate from a general simple result of the last two authors (see \cite{KupSag}, Theorem~9). However, this seems to be more of a coincidence and does not work out in general. An embedding $\mathcal{T}(2,3,4) \subset \B(2,2)\times\B(3)$ along with Theorem~\ref{th chi grid} gives $\chi_{\mathcal{T}(2,3,4)} \ge 3/2$, while the upper bound we can get is $\chi_{\mathcal{T}(2,3,4)} \le 7/4$. It is not clear for us which bound is closer to the truth. Let us mention that the `planar' version of this problem was solved for all non-degenerate triangles in the subsequent paper~\cite{NatSag}.

Recall that, as we observed in the proof of Proposition~\ref{prop upper for infinite intro}, the requirement on $\M$ to be infinite is not really needed in order to prove that $\chi(\R_{\infty}^n,\M) \le n+1$. Instead, it is sufficient to have $|\M| \ge n^{n+o(n)}$. It would be interesting to improve this bound and find the smallest value $f(n)$ such that $\chi(\R_{\infty}^n,\M) \le n+1$ whenever $|\M|\ge f(n)$. The best lower bound we have is $f(n) \ge \frac{2^n}{n+1}$. Indeed, if $\M$ consists of  fewer than $\frac{2^n}{n+1}$ points at unit distance apart, then  we clearly need at least $n+2$ colours even for a colouring of $\{0,1\}^n$ without a monochromatic copy of $\M$. Note that the finitude of the Euclidean counterpart of this $f(n)$, i.e. the existence of $g(n) \in \N$ such that $\chi(\R_2^n,\M)=2$ provided $|\M|\ge g(n)$, is an old problem of Erd\H{o}s which remains open even for $n=2$, see \cite[Section~3]{ConFox}. 

Some of the most exciting open problems concern infinite sets. The first problem here is to  improve the logarithmic lower bound in Theorem~\ref{thminf2 intro}. We conjecture that  for all sufficiently small $q$, the value $\chi(\R_{\infty}^n,\g(q))$ equals exactly $n+1$ for all $n \in \N$.
It follows from Theorem~\ref{thminf2 intro} that geometric progressions $\g(q)$ are Ramsey for all sufficiently small $q$ (namely, for all $q < 1/32$) since the value $\chi(\R_{\infty}^n,\g(q))$ grows with at least logarithmic speed. On the other hand, it was shown in the subsequent paper~\cite{KirSag} that $\chi(\R_{\infty}^n,\g(q))=2$ if $q$ is sufficiently close to $1$ (in terms of $n$). We wonder what is the `threshold' for $\g(q)$ to be Ramsey. 

Finally, it would be interesting to study related questions in other normed spaces. For a given norm $N$ in $\R^n$, what is the largest chromatic number $\chi(\R^n_N, \M)$, where $\M \subset \R_{N}^n$ is infinite? Can this maximum exceed $n+1$ for some norm? A partial answer to these questions in case $n=2$ was obtained in a subsequent paper~\cite{FGST}. 

\section*{Acknowledgments} Part of this work was done while the first author was supported by an LMS Early Career Fellowship. The first  and the third authors were also partially supported by ERC Advanced Grant "GeoScape" No. 882971. The second author is supported by the \href{https://rscf.ru/project/21-71-10092/}{RSF grant N 21-71-10092}. The third author is a winner of August M\"obius Contest and Young Russian Mathematics Contest and would like to thank its sponsors and jury.

\end{document}